\documentclass[a4paper,10pt]{amsart}

\usepackage[english]{babel}

\usepackage{lmodern}

\usepackage[utf8]{inputenc}

\usepackage[babel=true]{csquotes}

\usepackage{comment}

\usepackage{amsfonts}
\usepackage{amsthm} 
\usepackage{amsmath}
\usepackage{amssymb}

\usepackage{multirow} 
\usepackage{kbordermatrix}

\usepackage{brunnian}

\usetikzlibrary{arrows}
\usetikzlibrary{matrix}

\theoremstyle{definition}
\newtheorem{definition}{Definition}[section] 

\theoremstyle{plain}
 
\newtheorem{proposition}[definition]{Proposition}
\newtheorem{theorem}[definition]{Theorem}
\newtheorem{corollary}[definition]{Corollary}

\newtheorem*{maintheo}{Main Theorem}
\newtheorem*{theointro}{Theorem}

\theoremstyle{remark}
\newtheorem{remark}[definition]{Remark}
\newtheorem{example}[definition]{Example}
\newtheorem{question}[definition]{Question}

\newcommand{\C}{\mathbb{C}}
\newcommand{\R}{\mathbb{R}}
\newcommand{\Q}{\mathbb{Q}}
\newcommand{\Z}{\mathbb{Z}}
\newcommand{\N}{\mathbb{N}}
\newcommand{\F}{\mathbb{F}}

\newcommand{\NN}{\mathcal{N}}

\title{Gluing formulas for the $L^2$-Alexander torsions}
\author{Fathi Ben Aribi}

\address{Universit\'e de Gen\`eve, Section de math\'ematiques, 2-4 rue du Li\`evre, Case postale 64 1211 Gen\`eve 4, Suisse}

\email{fathi.benaribi@unige.ch}
\begin{document}

\renewcommand{\proofname}{Proof}


\maketitle

\begin{abstract}
We prove a Torres-like formula for the $L^2$-Alexander torsions of links, as well as formulas for connected sums and cablings of links. Along the way we compute explicitly the $L^2$-Alexander torsions of torus links inside the three-sphere, the solid torus and the thickened torus.
\end{abstract}

\section{Introduction}

When one considers a knot invariant, one of the first questions to ask is ``how can the definition be extended to links?''.
The $L^2$-Alexander invariant of a knot is a knot invariant taking values in the class of maps on the positive real numbers up to multiplication by monomials. 
It was originally constructed by W. Li and W. Zhang in \cite{LZ06a} from a presentation of the knot group, as an infinite-dimensional version of Fox's construction for the Alexander polynomial (see \cite{fox}).
J. Dubois, S. Friedl and W. L\"uck then generalized this invariant in \cite{DFL} with the $L^2$-Alexander torsion
$T^{(2)}(M,\phi,\gamma)(t)$
 associated to a triplet $(M,\phi,\gamma)$ and a positive number $t$, where $M$ is a compact connected oriented $3$-manifold with empty or toroidal boundary, 
 $\phi\colon \pi_1(M) \to \Z$ is a group homomorphism (and can be seen as a $1$-cohomology class of $M$) 
 and $\gamma\colon \pi_1(M) \to G$ is a second homomorphism to a finitely presented group $G$ such that $\phi$ factors through $\gamma$.
One can see these $L^2$-Alexander torsions as infinite-dimensional versions of the Reidemeister torsions of $3$-manifolds.
In particular, the $L^2$-Alexander torsions associated to the exterior of a link in the three-sphere answer the previous question concerning the $L^2$-Alexander invariant of knots.

The $L^2$-Alexander torsions of links attract considerable interest for the reasons that they are continuous functions (see \cite{Liu}) whose values give the simplicial volume and the Thurston norm associated to 
the link exterior (see \cite{DFL,FL,Liu,LS99}).
Unfortunately, these torsions are difficult to compute exactly; thus, any method of computing them more efficiently for a given link is worthy of pursuit. 
For knots, the torsions reduce to the $L^2$-Alexander invariant of Li and Zhang when $\gamma$ is the identity. Exact computations have been made for torus knots in \cite{DW} and then for iterated torus knots in \cite{BA13}. More generally, the $L^2$-Alexander torsions of graph manifolds were computed explicitly (and reduced to the Thurston norm) in \cite{DFL, herrmannarxiv} and the ones of fibered manifolds were computed partially in \cite{DFL}.
In this article, we present various techniques to relate the $L^2$-Alexander torsions of links related by 
\begin{itemize}
\item deleting a component,
\item connected sum,
\item cabling,
\end{itemize}
and we compute explicitly the torsions associated to iterated torus links.

The Torres formula for the Alexander polynomial of links (see \cite{Torres}) is a useful tool to compare the polynomials of two links differing by one component (for example two-component links and knots). The Main Theorem of this article presents a Torres-like formula for the $L^2$-Alexander torsions of links.
One can see the $L^2$-Alexander torsions of links as functions 
$T^{(2)}_{L,(n_1, \ldots, n_{c})}(\gamma)(t)$
of a link $L$, a finite set of integers $n_i$ (one for each component of $L$), a homomorphism $\gamma$ and a positive variable $t$; these functions are considered up to multiplication by monomials (equality up to such multiplication is denoted $\dot{=}$).
\begin{maintheo}[Theorem \ref{thm surgery forget}]
Let $L = L_1 \cup \ldots \cup L_c$ be a $c$-component link and $L'$ the link obtained by removing the last component $L_c$. Let $G_L$ and $G_{L'}$ denote the fundamental groups of their exteriors and $Q\colon G_{L} \to G_{L'}$ the epimorphism induced by removing of the component $L_c$.
For all $t>0$, all integers $n_1, \ldots, n_{c-1}$ and all appropriate $\gamma\colon G_{L'} \to G$,
the corresponding $L^2$-Alexander torsions satisfy:
$$T^{(2)}_{L',(n_1, \ldots, n_{c-1})}(\gamma)(t)  \ \dot{=} \ 
\dfrac{T^{(2)}_{L,(n_1, \ldots, n_{c-1}, 0)}(\gamma \circ Q)(t)}
{\max(1,t)^{|\mathrm{lk}(L_1,L_c) n_1 + \ldots + \mathrm{lk}(L_{c-1},L_c) n_{c-1}|}}.$$
\end{maintheo}

The Main Theorem can be seen as a particular case of a general Dehn surgery formula for the $L^2$-Alexander torsions (stated in Propositions \ref{prop surgery} and \ref{prop surgery links}), since removing a component of the link $L$ is the same as gluing a solid torus canonically to the corresponding boundary component of the exterior of $L$.
Another particular case of Dehn surgery yields an interesting connection between $L^2$-Alexander torsions of the Whitehead link and the $L^2$-Alexander invariants of the twist knots (see Theorem \ref{thm surgery twist}).

Like many of the results of this article, the surgery formulas are consequences of a general gluing formula for the $L^2$-Alexander torsions. This formula is stated in Proposition \ref{prop mayer vietoris}, and can be seen as a result of an application of the multiplicativity of the classical $L^2$-torsion (see \cite[Theorem 3.35 (1)]{Luc02}).

As a second class of consequences of the gluing formula of  Proposition \ref{prop mayer vietoris}, we prove general formulas for computing the $L^2$-Alexander torsions of  connected sums of two links or  general multi-component cablings of  links, as summarized in the following Theorem (for simplicity $\gamma$ is assumed to be the identity and is therefore not written):
\begin{theointro}[Theorems \ref{thm_torsion_sum} and \ref{thm_torsion_cabling}]

\

\begin{enumerate}
\item The $L^2$-Alexander torsions are ``almost multiplicative'' under the connected sums of links: if $L''$ is a connected sum of a $(c+1)$-component link $L$ and a $(d+1)$-component link $L'$, then
$$
\dfrac{
T^{(2)}_{L'',(n_1, \ldots, n_{c+d+1})}(t) 
}{ \max(1,t)^{|n_{c+d+1}|}}
\ \dot{=} \ 
T^{(2)}_{L,(n_1, \ldots,n_c, n_{c+d+1})}(t)
\cdot
T^{(2)}_{L',(n_{c+1}, \ldots, n_{c+d}, n_{c+d+1})}(t)
.$$
\item The $L^2$-Alexander torsions satisfy a cabling formula: if $L$ is a $(c+1)$-component link, and $L'$ a $(ep,eq)$-cabling on the last component of $L$, then
$$
T^{(2)}_{L', (n_1, \ldots, n_{c+e})}(t) 
\ \dot{=} \ 
T^{(2)}_{L, (n_1, \ldots,n_c, p N)}(t) \cdot
\max(1,t)^{(e|p|-1)|\ell+q N|}, 
$$
where $N = n_{c+1} + \ldots + n_{c+e}$ and $\ell = \sum_{i=1}^c \mathrm{lk}(L_i,L_{c+1}) n_i.$
\end{enumerate}
\end{theointro}
Formulas (1) and (2) are generalizations of Theorems 3.2 and 4.3 of \cite{BA13} to the case of links. 
Unlike \cite{BA13} where we used Fox calculus as the main tool, here we use methods of computation that rely on CW-complexes structures and the gluing formula  of Proposition \ref{prop mayer vietoris}. 
We hope that \cite{BA13} and the present article will provide $L^2$-Alexander enthusiasts with a wide array of methods of computations.

The proof of the previous Theorem uses the gluing formula of Proposition \ref{prop mayer vietoris} and the fact that in making a connected sum of two links or cabling one link 
one has to glue specific Seifert-fibered pieces to produce the corresponding link exteriors.
We are thus naturally led to explicitly compute all $L^2$-Alexander torsions of links whose exteriors are Seifert-fibered. This is the third and final class of applications of the gluing formula of  Proposition \ref{prop mayer vietoris} and they are interesting in their own right. These computations are summarized in the following Theorem (for simplicity $\gamma$ is assumed to be the identity and is therefore not written):

\begin{theointro}[All Propositions of Section \ref{sec:splicing}] 
Let $L$ be a link in $S^3$ with Seifert-fibered exterior. Then $L$ can be seen as a torus link  either in $S^3$, or in a solid torus or else in a thickened torus. Moreover:
\begin{itemize}
\item If $L=T(ep,eq)$, i.e. $L$ is an $e$-component torus link with $p,q$ coprime numbers,
 then
$$T^{(2)}_{L, (n_1, \ldots,n_e)}(t) \ \dot{=} \
\max(1,t)^{(e|pq| - |p| - |q|)|n_1 + \ldots +n_e|}.$$
\item If $L= T(ep,eq) \cup L_{e+1}$, i.e. $L$ is an $e$-component torus link in a solid torus (seen as the exterior of the unknot $L_{e+1}$),
 then
$$T^{(2)}_{L, (n_1, \ldots,n_e, n_{e+1})}(t) \ \dot{=} \
\max(1,t)^{(e|p|-1)|q(n_1 + \ldots +n_e)+ n_{e+1}|}.$$
\item If $L= T(ep,eq) \cup L_{e+1} \cup L_{e+2}$, i.e. $L$ is an $e$-component torus link in a thickened torus (seen as the exterior of the Hopf link $L_{e+1} \cup L_{e+2}$),
 then
$$ T^{(2)}_{L, (n_1, \ldots,n_e, n_{e+1}, n_{e+2})}(t) 
\ \dot{=} \
\max(1,t)^{e|pq(n_1 + \ldots +n_e) + p n_{e+1} + q n_{e+2}|}.
$$
\end{itemize}
\end{theointro}
Since the $L^2$-Alexander torsions of a Seifert-fibered manifold are known to be $(t \mapsto \max(1,t))$ to the power the Thurston norm of the cohomology class $\phi$ (see \cite{DFL,herrmannarxiv}), then
the previous Theorem
offers  a new way of computing  the Thurston norm for all links with Seifert-fibered exteriors. We hope this work can be of use to fellow topologists interested in Thurston norms of links.

Most of the results of this article come from the author's PhD thesis \cite{BAthesis}.
The article is organized as follows:
Section \ref{sec:def} reviews some well-known facts about knots, groups, and $L^2$-invariants, and can be skimmed by the experienced reader;
Section \ref{sec:first} reviews some important basic formulas satisfied by the $L^2$-Alexander torsions, notably the gluing formula;
Section \ref{sec:dehn} presents the various Dehn surgery formulas and the Main Theorem;
finally, Section \ref{sec:splicing} tackles the computation of the torsions of Seifert-fibered link exteriors as well as the connected sum and cabling formulas.

\section*{Acknowledgements}

I would like to thank my PhD advisor J\'er\^ome Dubois, for his teachings and great advice.
This article is based on work supported by the \textit{Minist\`ere de l'Enseignement Sup\'erieur et de la Recherche} at the Universit\'e Paris Diderot
and by the \textit{Swiss National Science Foundation}, subsidy $200021\_ 162431$, at the Universit\'e de Gen\`eve.

\section{Preliminaries}
\label{sec:def}

\subsection{Knots and links}

Here we follow mostly \cite{BZ}. We choose an orientation for the three-sphere $S^3$.
A \textit{link} with $c \in \N$ components is an embedding of a disjoint 
union of $c$ circles $\sqcup_{i=1}^c S^1$ into $S^3$; we will 
assume that all links have ordered oriented components. We consider links up to ambient isotopies in $S^3$ that preserve the order and the orientation 
of the components, unless precised otherwise. 
For $L = L_1 \cup \ldots \cup L_c$ an link in $S^3$, let $V(L)$ denote an open tubular neighbourhood of $L$, and $M_L = S^3 \setminus V(L)$ denote the \textit{exterior of $L$}, which is a compact $3$-manifold with toroidal boundary. 
The orientation of $M_L$ comes from the one of $S^3$, and does not depend on the orientation of $L$. Each boundary torus $\partial M_{L_i}$ is oriented with the convention that vectors normal to the boundary point outside of $M_L$. 
A \textit{split link} is a link $L \subset S^3$ such that there exists a $2$-sphere $\Sigma \subset S^3$,  $L=L' \sqcup L''$ with $L'$ and $L''$ sub-links, and $L'$ and $L''$ are contained in different connected components of $S^3 \setminus \Sigma$. Most of the time we will assume that links are non-split.
The \textit{group} of a link $L$ is the fundamental group of its exterior and is denoted $G_L = \pi_1(M_L)$. We denote $\alpha_L\colon G_L \twoheadrightarrow \Z^c$ the abelianization homomorphism, where $c$ is the number of components of $L$. The linking number between two components $L_i, L_j$ of a link $L$ is denoted $\mathrm{lk}(L_i,L_j)$.

A link with one component is called a \textit{knot}. When $K$ is an oriented knot, there exists, up to isotopy, a unique pair of simple closed curves $\mu_K$ and $\lambda_K$ on the $2$-torus $\partial M_K = \partial V(K)$ such that $\mu_K$ bounds a disk in $V(K)$ and $\lambda_K$ is homologous to $K$ in $V(K)$. We choose an orientation for these two curves such that the linking number between $\mu_K$ and $K$ and the intersection number between $\mu_K$ and $\lambda_K$ are both $+1$. The pair $(\mu_K,\lambda_K)$ is called a \textit{preferred meridian-longitude pair for $K$}. Any such $\mu_K$ is called a \textit{meridian curve}. Here we have used the notations and definitions of \cite{Tsau}.

\subsection{$L^2$-invariants}

We follow \cite{DFL} and \cite{Luc02} for the rest of Section \ref{sec:def}.
Given a countable discrete group~$G$, the completion of the algebra~$\C[G]$ endowed with the scalar product
$ \left \langle \sum_{g \in G} \lambda_g g , \sum_{g \in G} \mu_g g \right \rangle:= \sum_{g \in G} \lambda_g \overline{\mu_g}$
is the Hilbert space ~$$ \ell^2(G):= \left \{ \sum_{g \in G} \lambda_g g \ | \  \lambda_g \in \mathbb{C} , \sum_{g \in G} | \lambda_g |^2 < \infty \right \},$$ of square-summable complex functions on~$G$. We denote by~$B(\ell^2(G))$ the algebra of  operators on~$\ell^2(G)$ that are bounded with respect to the operator norm. 

Given~$h \in G$, we define the corresponding \textit{left-} and \textit{right-multiplication operators}~$L_{h}$ and~$R_{h}$  in~$B(\ell^2(G))$ as extensions of the automorphisms~$(g \mapsto hg)$ and~$(g \mapsto gh)$ of~$G$.
One can extend the operators~$R_{h}$~$\mathbb{C}$-linearly to an operator~$R_w\colon \ell^2(G) \to \ell^2(G)$ for any~$w \in \C[G]$. Moreover, if~$\ell^2(G)^n$ is endowed with its usual Hilbert space structure and~$A = \left ( a_{i,j} \right ) \in M_{p,q}(\C[G])$ is a~$\C[G]$-valued~$p\times q$ matrix, then the right multiplication
$$R_A:= \left (R_{a_{i,j}}\right )_{1 \leqslant i \leqslant p, 1 \leqslant j \leqslant q}$$ provides a bounded operator~$\ell^2(G)^{q} \rightarrow \ell^2(G)^{ p}$. Note that we shall consider elements of~$\ell^2(G)^n$ as \textit{column vectors} and suppose that matrices with coefficients in~$B(\ell^2(G))$ act on the \textit{left} (even if the coefficients are themselves \textit{right}-multiplication operators).

The \textit{von Neumann algebra}~$\mathcal{N}(G)$ of the group~$G$ is the sub-algebra of~$B(\ell^2(G))$ made up of~$G$-equivariant operators (i.e. operators that commute with all left multiplications~$L_h$). A \textit{finitely generated Hilbert~$\mathcal{N}(G)$-module} consists in a Hilbert space~$V$ together with a left~$G$-action by isometries such that there exists a positive integer~$m$ and an embedding~$\varphi$ of~$V$ into~$\ell^2(G)^m$. A \textit{morphism of finitely generated Hilbert~$\mathcal{N}(G)$-modules}~$f \colon U \rightarrow V$ is a linear bounded map which is~$G$-equivariant.

Denoting by~$e$ the neutral element of~$G$, the von Neuman algebra of~$G$ is endowed with the \textit{trace}~$\mathrm{tr}_{\mathcal{N}(G)} \colon \mathcal{N}(G)  \rightarrow \mathbb{C}, \phi \mapsto  \left \langle \phi (e) , e \right \rangle$ which extends to~\\
$\mathrm{tr}_{\mathcal{N}(G)} \colon M_{n,n}(\mathcal{N}(G)) \rightarrow \C$ by summing up the traces of the diagonal elements.

\begin{definition}
The \textit{von Neumann dimension} of a finitely generated Hilbert~$\mathcal{N}(G)$-module~$V$ is defined as
~$$\dim_{\mathcal{N}(G)}(V) := \mathrm{tr}_{\mathcal{N}(G)}(\mathrm{pr}_{\varphi(V)}) \in \R_{\geqslant 0},$$
where~$ \mathrm{pr}_{\varphi(V)} \colon  \ell^2(G)^m \to  \ell^2(G)^m~$ is the orthogonal projection onto~$\varphi(V)$.
\end{definition}

 The von Neumann dimension does not depend on the embedding of~$V$ into the finite direct sum of copies of~$\ell^2(G)$. 

\subsection{The Fuglede-Kadison determinant}

The \textit{spectral density}~$F(f)$ of a morphism~$f \colon U \to V$ of finitely generated Hilbert~$\mathcal{N}(G)$-modules maps~$\lambda \in \mathbb{R}_{\geqslant 0}$ to 
$$ F(f)(\lambda):= \sup \{
 \dim_{\mathcal{N}(G)} (L) | L \in \mathcal{L}(f,\lambda) \},$$
where~$\mathcal{L}(f,\lambda)$ is the set of finitely generated Hilbert~$\mathcal{N}(G)$-submodules of~$U$ on which the restriction of~$f$ has a norm smaller or equal to~$\lambda$. Since~$F(f)(\lambda)$ is monotonous and right-continuous, it defines a measure~$dF(f)$ on the Borel set of~$\R_{\geqslant 0}$ solely determined by the equation~$dF(f)(]a,b]) = F(f)(b)-F(f)(a)$ for all~$a<b$.

\begin{definition} \label{def detFK}
The \textit{Fuglede-Kadison determinant of $f$} is defined by:
\begin{equation*}\label{detFK}
{det}_{\mathcal{N}(G)}(f):= \exp \left ( \int_{0^+}^\infty \ln(\lambda) \, dF(f)(\lambda) \right )
\end{equation*}
if $\int_{0^+}^\infty \ln(\lambda) \, dF(f)(\lambda) > -\infty$\,; if not, $det_{\mathcal{N}(G)}(f) = 0$. 

When  $\int_{0^+}^\infty \ln(\lambda) \, dF(f)(\lambda) > -\infty$, we say that \emph{$f$ is of determinant class}.
\end{definition}
Here are two properties of the determinant we will use in the rest of this article (see \cite{Luc02} for more details and proofs).

\begin{proposition} \label{prop operations det}
Let $G$ be a countable discrete group.
\begin{enumerate}
\item For all $f,g$ morphisms of finitely generated Hilbert $\mathcal{N}(G)$-modules, $$
det_{\mathcal{N}(G)} \left (
\begin{pmatrix}
f & 0 \\ 0 & g
\end{pmatrix}
\right ) = det_{\mathcal{N}(G)}(f) \cdot det_{\mathcal{N}(G)}(g).$$
\item For all $t \in \C$, if $g \in G$ has infinite order, then $ Id - t R_g$ is injective and 
$$
det_{\mathcal{N}(G)} ( Id - t R_g) = \max ( 1 , |t|).
$$
\end{enumerate}
\end{proposition}

\subsection{$L^2$-torsion} \label{$L^2$-torsion}

A \textit{finite Hilbert~$\NN(G)$-chain complex}~$C_*$ is a sequence of morphisms of finitely generated Hilbert~$\NN(G)$-modules
$$C_* = \left (0 \to C_n \overset{\partial_n}{\longrightarrow} C_{n-1} 
\overset{\partial_{n-1}}{\longrightarrow} \ldots
\overset{\partial_2}{\longrightarrow} C_1 \overset{\partial_1}{\longrightarrow} C_0 \to 0\right )$$
such that~$\partial_p \circ \partial_{p+1} =0$ for all~$p$.
The \textit{$p$-th~$L^2$-homology} of such a chain complex~$C_*$ is the finitely generated Hilbert~$\NN(G)$-module
$$H_p^{(2)}(C_*) := \textrm{Ker}(\partial_p) / \overline{\textrm{Im}(\partial_{p+1})}.$$
The \textit{$p$-th~$L^2$-Betti number of~$C_*$} is defined as~$b_p^{(2)}(C_*) := \dim_{\NN(G)}(H_p^{(2)}(C_*))$. A finite Hilbert~$\NN(G)$-chain complex~$C_*$ is \textit{weakly acyclic} if its~$L^2$-homology is trivial  (i.e. if all its~$L^2$-Betti numbers vanish) and of \textit{determinant class} if all the operators~$\partial_p$ are of determinant class. 

\begin{definition}
Let $C_*$ be a finite Hilbert $\NN(G)$-chain complex as above. 
We define its \textit{$L^2$-torsion} by
$$T^{(2)}(C_*) := \prod_{i=1}^n \det {}_{\NN(G)}(\partial_i)^{(-1)^i} \in \R_{>0}$$
when $C_*$ is weakly acyclic and of determinant class, and by~$T^{(2)}(C_*)=0$ otherwise.
\end{definition}

The following proposition will be useful for computations of $L^2$-torsions. Compare with \cite[Proposition 1.58]{BAthesis} and \cite[Lemma 3.1]{DFL}.

\begin{proposition} \label{prop tau chaine}
Let 
$C_* = \left (0 \to 
\ell^2(G)^k
\overset{\partial_2}{\longrightarrow} 
\ell^2(G)^{k+l}
 \overset{\partial_1}{\longrightarrow} 
\ell^2(G)^l
  \to 0\right )$
   be a $2$-dimensional finite Hilbert $\NN(G)$-chain complex and let $J \subset \{1, \ldots,k+l\}$ be a subset of $\{1, \ldots,k+l\}$ of size $l$.
For $i=1,2$, $\partial_i$ is naturally written as a matrix with coefficients operators in $B(\ell^2(G))$. We write $\partial_1(J): \ell^2(G)^l \to \ell^2(G)^l$ the operator composed of the columns of $\partial_1$ indexed by $J$, and $\partial_2(J): \ell^2(G)^k \to \ell^2(G)^k$ the operator obtained from $\partial_2$ by deleting the rows indexed by $J$.
If $\partial_2(J)$ and $\partial_1(J)$ are injective and of determinant class, then $C_*$ is weakly acyclic and of determinant class, and
$$T^{(2)}(C_*)  = \dfrac{\det_{\NN(G)}(\partial_2)}{\det_{\NN(G)}(\partial_1)} = \dfrac{\det_{\NN(G)}(\partial_2(J))}{\det_{\NN(G)}(\partial_1(J))}.$$
\end{proposition}

There exists an immediate generalisation of Proposition \ref{prop tau chaine} to any dimension that mirrors the formula of \cite[Theorem 2.2]{Turaev}.

\subsection{$L^2$-Alexander torsions}

We follow the definitions and notations of \cite{DFL}.
Let $\pi$ be a group and $\phi\colon \pi \to \Z$, $\gamma\colon \pi \to G$ two group homomorphisms.
We say that $(\pi,\phi,\gamma)$ \textit{forms an admissible triple} if $\phi \colon  \pi \to \Z$ factors 
through $\gamma$ (i.e. there is a group homomorphism $\psi \colon G 
\to \Z$ such that $ \phi = \psi \circ \gamma$, see the diagram below).

\begin{center}
\begin{tikzpicture}
\begin{scope}[xshift=7cm,rotate=0,scale=1]
[description/.style={fill=white,inner sep=2pt}] 
\matrix(a)[matrix of math nodes, row sep=1em, column sep=3.5em, text height=1.5ex, text depth=0.25ex] 
{ \pi & G \\ \ & \Z\\}; 
\path[->](a-1-1) edge node[above]{$\gamma$} (a-1-2);
\path[->](a-1-1) edge node[below]{$\phi$} (a-2-2);
\path[->, dotted](a-1-2) edge  (a-2-2);
\end{scope}
\end{tikzpicture}
\end{center}

For $X$ a CW-complex, we say that $(X,\phi\colon \pi_1(X) \to \Z,\gamma\colon \pi_1(X) \to G)$ \textit{forms an admissible triple} if $(\pi_1(X),\phi,\gamma)$ forms one.
Let $(X,\phi,\gamma)$ be such an admissible triple, $\pi= \pi_1(X)$ and $t>0$. We define a ring homomorphism
$$\kappa(\pi, \phi, \gamma, t)\colon \begin{pmatrix}
 & \Z[\pi]& \longrightarrow & \R[G] \\
 & \sum_{j=1}^r m_j g_j & \longmapsto & \sum_{j=1}^r m_j t^{\phi(g_j)} \gamma(g_j)
\end{pmatrix}
$$
and we also denote $\kappa(\pi, \phi, \gamma, t)$ its induction over the $M_{p,q}(\Z[\pi])$.

Assume $X$ is compact. The cellular chain complex of $\widetilde{X}$ denoted
$C_*(\widetilde{X},\Z) = $ \\
$\left (\ldots \to \bigoplus_i \Z[\pi] \widetilde{e}_i^k \to \ldots\right )$
is a chain complex of left $\Z[\pi]$-modules.
Here the $\widetilde{e}_i^k$ are lifts of the cells $e_i^k$ of $X$.
The group $\pi$ acts on the right on $\ell^2(G)$ by $g \mapsto R_{\kappa(\pi, \phi, \gamma, t)(g)}$, an action which induces a structure of right $\Z[\pi]$-module on $\ell^2(G)$.
Let
$$C_*^{(2)}(X,\phi,\gamma,t) = \ell^2(G) \otimes_{\Z[\pi]}  C_*(\widetilde{X},\Z)$$
 denote the finite Hilbert $\NN(G)$-chain complex obtained by the tensor product associated to these left- and right-actions; we call $C_*^{(2)}(X,\phi,\gamma,t)$ \textit{a $\NN(G)$-cellular chain complex of $X$}. 

\begin{definition}
If $C_*^{(2)}(X,\phi,\gamma,t)$ is a $\NN(G)$-cellular chain complex of $X$, then denote
$$ T^{(2)}(X,\phi,\gamma)(t) = T^{(2)}\left (C_*^{(2)}(X,\phi,\gamma,t)\right )$$
the \textit{$L^2$-Alexander torsion of $(X,\phi,\gamma)$ at $t>0$}. It is non-zero if and only if $C_*^{(2)}(X,\phi,\gamma,t)$ is weakly acyclic and of determinant class
\end{definition}

\begin{remark}
As a consequence of \cite{L}, if $X$ is a compact connected oriented irreducible $3$-manifold with empty or toroidal boundary and infinite fundamental group $\pi = \pi_1(X)$, then 
 $C_*^{(2)}(X,\phi,id,t)$ is weakly acyclic and of determinant class for all $\phi$ and all $t$.
In particular,  most of the statements in this article concern torsions of exteriors of non-split links in the three-sphere, and thus one can skip the assumptions on weak acyclicity and determinant class for $\gamma = id$ in such statements.
\end{remark}

We will  consider the equivalence class of $(t \mapsto T^{(2)}(X,\phi,\gamma,t))$ up to multiplication by the $(t \mapsto t^m), m \in \Z$, which does not depend on the CW-structure chosen on $X$ (the technical details can be found in \cite{BAthesis}). For two maps $f,g\colon \R_{>0} \to \R_{>0}$, we write 
$$ f \ \dot{=} \ g \ \Longleftrightarrow \ \exists m \in \Z, \forall t>0, f(t) = t^m g(t).$$
Note that for $t>0$ and any integer $k$, $\max(1,t^k) = t^{\frac{k-|k|}{2}} \max(1,t)^{|k|}$. Thus 
$\left (t \mapsto \max(1,t^k)\right ) \ \dot{=} \ \left (t \mapsto \max(1,t)^{|k|}\right )$.

For $X$ a CW-complex, its \textit{$L^2$-torsion}  is defined as
$ T^{(2)}(X) := T^{(2)}(X,0,id)(1)$
when $C_*^{(2)}(X,0,id,1)$ is weakly acyclic and of determinant class. Of course, historically, $L^2$-torsions came before $L^2$-Alexander torsions (see \cite[Section 3.4]{Luc02}).
The following astonishing theorem of W. L\"uck and T. Schick (see \cite{LS99}) states that the $L^2$-torsion of an irreducible $3$-manifold gives precisely the simplicial volume of this manifold. 
Recall that a compact connected orientable $3$-manifold $M$ is called \textit{hyperbolic} if its interior admits a complete Riemannian metric whose sectional curvature is constant equal to $-1$, and \textit{Seifert} if it admits a foliation by circles.

\begin{theorem}[\cite{Luc02}, Theorem 4.3] \label{thm t2 volume}
Let $M$ be a compact connected orientable irreducible $3$-manifold with infinite fundamental group and empty or incompressible toroidal boundary.
According to Thurston-Perelman's Geometrization theorem, $M$ splits along disjoint incompressible tori into pieces that are Seifert manifolds or hyperbolic manifolds with finite volume. 
Moreover, $C_*^{(2)}(M,0,id,1)$ is weakly acyclic and of determinant class, and
$$T^{(2)}(M) = \exp\left (\dfrac{\textrm{vol}(M)}{6 \pi}\right )$$
where $\textrm{vol}(M) = \sum_{i=1}^h \textrm{vol}_{\textrm{hyp}}(M_i)$ is the simplicial volume of $M$, defined as the sum of the hyperbolic volumes of the hyperbolic pieces $M_1 , \ldots , M_h$.
\end{theorem}

\subsection{$L^2$-Alexander torsions of links}
Let $L = L_1 \cup \ldots \cup L_c$ be a link in $S^3$, $M_L$ its exterior and $\alpha_L \colon G_L \to \Z^c$ the abelianization of its group. Any homomorphism $\phi\colon G_L \to \Z$ factors through $\alpha_L$ and thus is written
$\phi = (n_1, \ldots,n_c) \circ\alpha_L$
where $n_1, \ldots,n_c \in \Z$.
Any admissible triple $(M_L,\phi,\gamma)$ can thus be written 
$(M_L,(n_1, \ldots,n_c) \circ\alpha_L,\gamma)$, and we will therefore denote
$$T^{(2)}_{L,(n_1, \ldots,n_c)}(\gamma)(t)
:= T^{(2)}(M_L,(n_1, \ldots,n_c) \circ\alpha_L,\gamma)(t)$$
the \textit{$L^2$-Alexander torsion associated to the multi-link $(L,(n_1, \ldots,n_c))$ and the morphism $\gamma$ at the value $t$},
and $T^{(2)}_{L,(n_1, \ldots,n_c)}
:= T^{(2)}(M_L,(n_1, \ldots,n_c) \circ\alpha_L,id)$ the \textit{full $L^2$-Alexander torsion function associated to the multi-link $(L,(n_1, \ldots,n_c))$}.

Note that if $L=K$ is a knot and $n\in \Z$, then the full $L^2$-Alexander torsion of $K$ reduces to its $L^2$-Alexander invariant $\Delta_{K}^{(2)}$ (see \cite{DFL} for details):
$$T^{(2)}_{K,n}(t) \ \dot{=} \ 
\dfrac{\Delta_{K}^{(2)}(t^n)}{\max(1,t)^{|n|}}.$$

\section{First formulas}
\label{sec:first}

We present three useful formulas, which are direct generalizations of known properties of the $L^2$-torsions. The details of the proofs
of Propositions \ref{prop L2 torsion simple homotopy}
and \ref{prop mayer vietoris}
 can be found in \cite{BAthesis}.

\subsection{Simple homotopy equivalence}

The following proposition states that the $L^2$-Alexander torsions are invariant by simple homotopy equivalence.
Note that this result is a direct generalization of the formulas of \cite[Corollary 9.2]{Turaev} and \cite[Theorem 3.96 (1)]{Luc02}, and was announced in \cite{DFL}.
The technical details of the proof can be found in \cite[Theorem 2.12]{BAthesis}.

\begin{proposition}[\cite{BAthesis}, Theorem 2.12] \label{prop L2 torsion simple homotopy}
Let $f\colon X \to Y$ be a simple homotopy equivalence between two finite CW-complexes that induces the group isomorphism $f_*\colon \pi_1(X) \to \pi_1(Y)$. The triple
$(Y, \phi, \gamma)$ is an admissible triple if and only if $(X, \phi \circ f_*, \gamma \circ f_*)$  is one, 
the $\NN(G)$-cellular chain complex 
$C_*^{(2)}(X, \phi \circ f_*, \gamma \circ f_*,t)$ is weakly acyclic and of determinant class if and only if 
$C_*^{(2)}(Y, \phi, \gamma,t)$ is, and
$$T^{(2)}(X, \phi \circ f_*, \gamma \circ f_*)(t) \ \dot{=} \
T^{(2)}(Y, \phi, \gamma)(t).$$
\end{proposition} 

If $N$ is a compact smooth $3$-manifold, then it follows from theorems due to Chapman and Cohen (see \cite{chapman,cohen}) that any two CW-structures on $N$ are simple homotopy equivalent. Proposition \ref{prop L2 torsion simple homotopy} thus implies that for any admissible triple $(\pi_1(N), \phi, \gamma)$, $T^{(2)}(N,\phi,\gamma)$ is a well-defined topological invariant of $N$.

\subsection{CW-complexes of the form $W \times S^1$}

Let $W$ be a finite CW-complex, $S^1$ the one-dimensional circle with its classical CW-complex structure (one $0$-cell and one $1$-cell). Let $X = W \times S^1$ be the product space , whose CW-structure is induced by the direct product.
We prove how to relate the $L^2$-Alexander torsions of $X$ to the Euler characteristic $\chi(W)$ of $W$, as a generalisation of \cite[Theorem 3.93 (4)]{Luc02}. Let $\pi_W = \pi_1(W)$, $T$ a fixed generator of $\pi_1(S^1) \cong \Z$ and $\pi_X = \pi_1(X) \cong \pi_W \times \Z$. We consider the natural inclusion $S^1 \hookrightarrow S^1 \times \{p_W\} \subset X$ (where $p_W$ is the chosen basepoint on $W$) and $i\colon \pi_1(S^1) \to \pi_X$ the induced group homomorphism.

\begin{proposition} \label{prop WS1}
Let $t>0$ and let $\phi: \pi_X \longrightarrow \Z$ and $ \gamma: \pi_X \to G$ such that $(\pi_X,\phi,\gamma)$ is an admissible triple. 
If $(\gamma \circ i)(T)$ has infinite order in $G$, then  
$C^{(2)}_*(X,\phi, \gamma,t)$ is weakly acyclic and of determinant class, and
$$T^{(2)}(X,\phi, \gamma)(t)
\ \dot{=} \ \max(1,t)^{- \chi(W)  |(\phi \circ i)(T)|}.$$
\end{proposition}

\begin{proof}
Let $n$ be the dimension of the CW-complex $W$, and $c_k$ the number of $k$-cells of $W$ of dimension $k$, for $0 \leqslant k \leqslant n$. One has immediately 
$\chi(W) = \sum_{k=0}^n (-1)^k c_k.$
Let us denote $e^k_i$ the $k$-cells of $W$, and $P,a$ the $0$-cell and the $1$-cell of $S^1$.
We consider the cellular chain complex of $\Z[ \pi_X]$-modules $C_*(\widetilde{X})$.
We fix lifts $\widetilde{e^k_i \times P}$ and  $\widetilde{e^{k-1}_j \times a}$ as bases of the $\Z[ \pi_X]$-modules (with $ 0 \leqslant i \leqslant c_k$, $0 \leqslant j \leqslant c_{k-1}$).

If we denote the boundary of the cells in $C_*(\widetilde{W})$ in the following way: 
$$\partial \left (\widetilde{e^k_i}\right ) = \sum_{j=1}^{c_{k-1}} \lambda_{k,i,j} \ g_{k,i,j} \ \widetilde{e^{k-1}_j}$$
where $\lambda_{k,i,j} \in \C, g_{k,i,j} \in \pi_W$,
then the boundary operators in $C_*(\widetilde{X})$ act as:
$$\partial \left (\widetilde{e^k_i \times P}\right ) = \sum_{j=1}^{c_{k-1}} \lambda_{k,i,j} \ g_{k,i,j} \ \widetilde{e^{k-1}_j \times P},$$
$$\partial \left (\widetilde{e^k_i \times a}\right ) =
(i(T)-1) \widetilde{e^k_i \times P} +
 \sum_{j=1}^{c_{k-1}} \lambda_{k,i,j} \ g_{k,i,j} \ \widetilde{e^{k-1}_j \times a}.$$
Thus the boundary operators are of the matricial form:
$$ \kbordermatrix{
 & \ \ \ \  & \widetilde{e^k_i \times P} \ \ \ \ &  \vrule   & & \widetilde{e^{k-1}_j \times a} & \\
 & & & \vrule & i(T)-1 & & 0 \\
\widetilde{e^{k-1}_j \times P} & & * \ \ \  & \vrule & & \ddots & \\
 & & & \vrule & 0 & & i(T)-1 \\
 \hline
 & & & \vrule & & & \\
\widetilde{e^{k-2}_l \times a} & & 0 & \vrule & & * & \\
 & & & \vrule & & & 
},$$
where the square upper right block is of size $c_{k-1}$.

From the generalisation of Proposition \ref{prop tau chaine} to any dimension, one can compute $T^{(2)}(X,\phi, \gamma)(t)$ from the Fuglede-Kadison determinants of the corresponding upper right block operators, which are injective and of determinant class as long as \\
$(\gamma \circ i)(T)$ has infinite order in $G$. The formula is then a consequence of Proposition \ref{prop operations det} (1) and (2) and the fact that
$\chi(W) = \sum_{k=0}^n (-1)^k c_k.$
\end{proof}

As consequences of 
Proposition \ref{prop WS1} we compute the $L^2$-Alexander torsions of the solid torus and the $2$-torus.

\begin{corollary} \label{torsion solid torus}
For $c$  a generator of $\pi_1(S^1 \times D^2)$, if $\gamma(c)$ is of infinite order in $G$, then $C_*^{(2)}(S^1 \times D^2, \phi, \gamma)(t)$ is weakly acyclic and of determinant class for all $t>0$, and its $L^2$-Alexander torsion is
$$T^{(2)}(S^1\times D^2, \phi, \gamma)(t) \ \dot{=} \ \dfrac{1}{\max(1,t)^{|\phi(c)|}}.$$
\end{corollary}

\begin{corollary} \label{torsion torus}
If $\gamma(\pi_1(S^1 \times S^1))$ is infinite, 
then
$C_*^{(2)}(S^1 \times S^1, \phi, \gamma)(t)$ is weakly acyclic and of determinant class for all $t>0$, and its $L^2$-Alexander torsion is
$$T^{(2)}(S^1\times S^1, \phi, \gamma)(t) \ \dot{=} \ 1.$$
\end{corollary}

\begin{proof}
Let $X = S^1 \times S^1$. Its group $\pi_1(X) \cong  \Z^2$ admits a presentation of the form $\langle T,S | TS=ST \rangle$. Since $\gamma(\pi_1(X))$ is infinite, then at least one of $T,S$ has an image by $\gamma$ of infinite order in $G$, for instance $T$. Then $X$ is homeomorphic to $S^1 \times W$ where $W$ is the circle corresponding to $S$. Proposition \ref{prop WS1} concludes the proof. 
\end{proof}

\subsection{Gluing formulas}

Let $X, A, B, V$ be compact connected topological spaces, such that $X = A \cup B$ and $V= A \cap B$. Assume that these four spaces are endowed with structures of finite CW-complexes such that the inclusions $ V \overset{I_A}{\hookrightarrow} A$, $V \overset{I_B}{\hookrightarrow} B$, $A \overset{J_A}{\hookrightarrow} X$, $B \overset{J_B}{\hookrightarrow} X$ and $V \overset{I}{\hookrightarrow} X$ all map a $k$-cell to a $k$-cell (which means that the CW-structure of $X$ is constructed from those of $A$ and $B$), and such that $I = J_A \circ I_A = J_B \circ I_B$. 
Let us denote  $ \pi_V \overset{i_A}{\to} \pi_A$, $\pi_V \overset{i_B}{\to} \pi_B$, $\pi_A \overset{j_A}{\to} \pi_X$, $\pi_B \overset{j_B}{\to} \pi_X$ and $\pi_V \overset{i}{\to} \pi_X$ the group homomorphisms induced by $I_A, I_B, J_A, J_B, I$. Remark that $i = j_A \circ i_A = j_B \circ i_B$.
These numerous maps are all written on a diagram below for clarity.

\begin{tikzpicture}
[description/.style={fill=white,inner sep=2pt}] 
\matrix(a)[matrix of math nodes, row sep=2em, column sep=2.5em, text height=1.5ex, text depth=0.25ex] 
{ & A\\ V & & X\\ & B\\}; 
\path[->](a-2-1) edge node[below]{$I_B$} (a-3-2); 
\path[->](a-2-1) edge node[above]{$I_A$} (a-1-2);  
\path[->](a-1-2) edge node[above]{$J_A$} (a-2-3); 
\path[->](a-3-2) edge node[below]{$J_B$} (a-2-3); 
\path[->](a-2-1) edge node[above]{$I$} (a-2-3);  
\begin{scope}[xshift=6cm,rotate=0,scale=1]
[description/.style={fill=white,inner sep=2pt}] 
\matrix(a)[matrix of math nodes, row sep=2em, column sep=2.5em, text height=1.5ex, text depth=0.25ex] 
{ & \pi_1(A)\\ \pi_1(V) & & \pi_1(X) & G\\ & \pi_1(B) & & \Z\\}; 
\path[->](a-2-1) edge node[below]{$i_B$} (a-3-2); 
\path[->](a-2-1) edge node[above]{$i_A$} (a-1-2);  
\path[->](a-1-2) edge node[above]{$j_A$} (a-2-3); 
\path[->](a-3-2) edge node[below]{$j_B$} (a-2-3); 
\path[->](a-2-1) edge node[above]{$i$} (a-2-3);
\path[->](a-2-3) edge node[above]{$\gamma$} (a-2-4);
\path[->](a-2-3) edge node[below]{$\phi$} (a-3-4);
\path[->, dotted](a-2-4) edge  (a-3-4);
\end{scope}
\end{tikzpicture}

\begin{proposition}[Gluing formula] \label{prop mayer vietoris}
Let $(\pi_X,\phi\colon \pi_X \to \Z,\gamma\colon \pi_X \to G)$ be an admissible triple, and $t>0$. 
If the three $\NN(G)$-cellular chain complexes 
$$C^{(2)}_*(V,\phi \circ i, \gamma \circ i,t), \ C^{(2)}_*(A,\phi \circ j_A, \gamma \circ j_A,t), \
C^{(2)}_*(B,\phi \circ j_B, \gamma \circ j_B,t)$$ are
weakly acyclic and of determinant class, then
  $C^{(2)}_*(X,\phi, \gamma,t)$ is weakly acyclic and of determinant class as well, and
$$T^{(2)}(X,\phi, \gamma)(t)
 \ \dot{=} \ \dfrac{T^{(2)}(A,\phi \circ j_A, \gamma \circ j_A)(t) \cdot 
T^{(2)}(B,\phi \circ j_B, \gamma \circ j_B)(t)}{T^{(2)}(V,\phi \circ i, \gamma \circ i)(t)}.
$$
\end{proposition}
This result is a direct generalization of \cite[Theorem 3.35 (1)]{Luc02} and was announced in \cite{DFL}.
The technical details of the proof can be found in \cite[Theorem 3.1]{BAthesis}. 
One can apply the gluing formula to the particular case of toroidal gluings of $3$-manifolds. The following result first appeared in \cite[Theorem 5.5]{DFL}, and we will illustrate how it can be seen as a consequence of Proposition \ref{prop mayer vietoris}. For details we refer to \cite[Proposition 4.1]{BAthesis}.

\begin{proposition}[Toroidal gluing formula] \label{prop t2 JSJ}
Let $N$ be a $3$-manifold and $\phi \in Hom(\pi_1(N);\Z)$. 
Let $T_1, \ldots, T_k$ be disjoint tori in $M$ and $N_1, \ldots, N_l$ the connected components of $M$ minus small tubular open neighbourhoods of the tori $T_i$.
For $i=1, \ldots, l$, we denote by $\iota_i: N_i \to N$ and $\tau_j: T_j \to N$ the inclusions. 
Let $t>0$ and let $\gamma: \pi_1(N) \to G$ be a homomorphism such that 
$(\pi_1(N),\phi,\gamma)$ is an admissible triple and
the restriction $\gamma \circ (\tau_j)_*$   to each $\pi_1(T_j)$ has infinite image. 
If $C_*^{(2)}(N_i, \phi \circ (\iota_i)_* , \gamma \circ (\iota_i)_*,t)$ is weakly acyclic and of determinant class for all $N_i$, then 
$C_*^{(2)}(N,\phi,\gamma,t)$ is weakly acyclic and of determinant class and 
$$T^{(2)}(N,\phi,\gamma)(t) \ \dot{=}\  \prod_{i=1}^l T^{(2)}(N_i, \phi \circ (\iota_i)_* , \gamma \circ (\iota_i)_*).$$
\end{proposition}

\begin{proof}
Let us first assume that $k=1$ and $l=2$. We apply Proposition \ref{prop mayer vietoris} with $A=N_1, B=N_2, V=T_1, X=N$. If we assume that
$(\gamma \circ (\tau_1)_*)(\pi_1(T_1))$ is infinite, then by Corollary \ref{torsion torus}, $C_*^{(2)}(T_1, \phi \circ (\tau_1)_* , \gamma \circ (\tau_1)_*,t)$ is weakly acyclic, of determinant class, and of $L^2$-torsion equal to $1$.
Besides, we assumed that $C_*^{(2)}(N_i, \phi \circ (\iota_i)_* , \gamma \circ (\iota_i)_*,t)$ is weakly acyclic and of determinant class for $i=1,2$. 
The result follows.
For bigger $k$ and $l$ one just applies the previous reasoning by induction on $k$, torus by torus. Note that rigorously speaking, the base points of the fundamental groups change at each step but this does not change the final formula. 
\end{proof}

\section{Dehn surgery formulas}
\label{sec:dehn}

We apply the gluing formula of Proposition \ref{prop mayer vietoris} to the case of Dehn surgery, where we glue a solid torus on a toroidal boundary component of a $3$-manifold. We start by recalling the definition of Dehn surgery.

\subsection{Dehn Surgery}

We follow \cite[Section 9F]{Rol}.
Let $M$ be a $3$-manifold and let $T_1, \ldots, T_n$ be $2$-tori that are connected components of $\partial M$. For each $i=1, \ldots, n$, specify a simple closed curve $J_i$ on each $T_i$. Let 
$$M' = M \cup_h \left ((S^1\times D^2) \sqcup \ldots \sqcup (S^1\times D^2)\right )$$
where $h$ is an union of homeomorphisms $h_i: S^1 \times S^1 \to T_i$, each of which taking a meridian curve $m_i$ of $\partial (S^1 \times D^2)$ to the curve $J_i$.
Up to homeomorphism $M'$ does not depend on the choice of $h$. We say that $M'$ is obtained by \textit{Dehn Filling on $M$}.
\textit{Dehn surgery} refers to the more general process of drilling out links in $M$ and then filling them along certain curves. We will mostly be interested in the Dehn filling process.

When $M$ is the exterior of an oriented link $L= L_1 \cup \ldots \cup \ldots L_n \cup L_{n+1} \cup \ldots \cup L_c$ in $S^3$ and $T_i = \partial V(L_i)$ for $i=1, \ldots, n$, each $L_i$ has a preferred meridian-longitude pair $(\mu_{i},\lambda_{i})$. We only need to specify the homotopy class of $J_i$ in $T_i$, described by two relatively prime integers $p_i, q_i$:
$$[J_i] = p_i [\mu_i] + q_i [\lambda_i].$$
We call $p_i/q_i \in \Q \cup \{ \infty \}$ the \textit{surgery coefficient} associated with the component $L_i$.

\begin{example}
A $p/q$-surgery on the trivial knot yields the lens space $L(p,q)$.

In particular, a $0$-surgery on the trivial knot yields $S^2 \times S^1$, and a $\pm 1/n$-surgery, $n\in \N$ on the trivial knot yields $S^3$.

A $\infty$-surgery on any knot yields $S^3$ (this corresponds to the canonical filling of the knotted tunnel).
\end{example}

\subsection{The Dehn Surgery formula}

Let $M$ be a compact connected $3$-manifold with non-empty toroidal boundary, $B$ a solid torus, $T$ a boundary part of $M$, and $J$ a simple closed curve on $T$. Let $N$ be the manifold obtained by doing a Dehn filling on $M$ for the curve $J$ on the boundary part $T$.
One has $N = M \cup B$ and $T= M \cap B$. Let $J'$ be a simple closed curve on $T$ such that the classes of $J$ and $J'$ form a system of generators of $\pi_1(T) \cong \Z^2$. We can assume that $J$ and $J'$ intersect on a single point $P$, which will be the base point for all the following fundamental groups.
We choose a CW-structure on $M$ and $T$ such that $P$ is a $0$-cell and $J$ and $J'$ are $1$-cells. For constructing the CW-structure of $B$ we choose a $2$-cell $D$ bounded by $J$, and a $3$-cell $\rho$ glued in the usual way to close the solid torus. Thus $J'$ and the core of $B$ have the same homotopy class in $\pi_1(B)$. We can thus see $J$ as a meridian of $B$ and $J'$ as a longitude of $B$. Finally we provide $N$ with the CW-structure composed of those of $M, T$ and $B$.
Let $\pi_M = \pi_1(M)$, $\pi_N = \pi_1\left (N\right )$ and $c$ the homotopy class of the core of $B$ in $\pi_1(B)$. Then the inclusion $J_M: M  \subset N$ induces a quotient group homomorphism $Q: \pi \twoheadrightarrow \pi_N$ (whose kernel is normally generated by $[J]$), and the inclusion $J_B: B \subset N$ induces a 
group homomorphism $\iota: c^{\Z} \to \pi_N$. The following diagram should make everything clearer.

\begin{tikzpicture}
[description/.style={fill=white,inner sep=2pt}] 
\matrix(a)[matrix of math nodes, row sep=2em, column sep=2.5em, text height=1.5ex, text depth=0.25ex] 
{ & M\\ T & & N\\ & B\\}; 
\path[->](a-2-1) edge node[below]{$I_B$} (a-3-2); 
\path[->](a-2-1) edge node[above]{$I_M$} (a-1-2);  
\path[->](a-1-2) edge node[above]{$J_M$} (a-2-3); 
\path[->](a-3-2) edge node[below]{$J_B$} (a-2-3); 
\path[->](a-2-1) edge node[above]{$I$} (a-2-3);  

\begin{scope}[xshift=6cm,rotate=0,scale=1]
[description/.style={fill=white,inner sep=2pt}] 
\matrix(a)[matrix of math nodes, row sep=2em, column sep=2em, text height=1.5ex, text depth=0.25ex] 
{ & \pi_1(M)\\ \pi_1(T) & & \pi_1\left (N\right ) & G\\ & \pi_1(B)\cong c^{\Z} & & \Z\\}; 
\path[->](a-2-1) edge node[below]{$i_B$} (a-3-2); 
\path[->](a-2-1) edge node[above]{$i_M$} (a-1-2);  
\path[->](a-1-2) edge node[above]{$Q$} (a-2-3); 
\path[->](a-3-2) edge node[below]{$\iota$} (a-2-3); 
\path[->](a-2-1) edge node[above]{$i$} (a-2-3);
\path[->](a-2-3) edge node[above]{$\gamma$} (a-2-4);
\path[->](a-2-3) edge node[below]{$\phi$} (a-3-4);
\path[->, dotted](a-2-4) edge  (a-3-4);
\end{scope}
\end{tikzpicture}

We can now state the general Dehn surgery formula for $L^2$-Alexander torsions:
\begin{proposition}[\cite{BAthesis}, Theorem 3.6] \label{prop surgery}
Let $\phi: \pi_N \to \Z$ and $\gamma: \pi_N \to G$ be group homomorphisms such that $(\pi_N,\phi,\gamma)$ forms an admissible triple. For all $t>0$, if $\gamma(\iota(c))$ is  of infinite order in $G$ and if
$C_*^{(2)}(M,\phi \circ Q,\gamma \circ Q)(t)$ is weakly acyclic and of determinant class, then 
$C_*^{(2)}(N,\phi,\gamma)(t)$ has the same properties and
$$
T^{(2)}(N,\phi,\gamma)(t) \ \dot{=} \ 
\dfrac{T^{(2)}(M,\phi \circ Q,\gamma \circ Q)(t)}
{\max(1,t)^{|\phi(\iota(c))|}}.
$$
\end{proposition}

\begin{proof}
Since $\gamma(\iota(c))$ is of infinite order in $G$,
 $C^{(2)}_*(B,\phi \circ \iota, \gamma \circ \iota,t)$ is weakly acylic and of determinant class by Corollary \ref{torsion solid torus}, and 
 $$T^{(2)}(B,\phi \circ \iota, \gamma \circ \iota)(t)  \ \dot{=} \  \dfrac{1}{\max(1,t)^{|\phi(\iota(c))|}}.$$

Likewise, $\gamma(i(\pi_1(T))) = \gamma(\iota(\pi_1(B)))$ is an infinite subgroup of $G$, thus, by Corollary \ref{torsion torus},
$C^{(2)}_*(T,\phi \circ i, \gamma \circ i,t)$ is weakly acylic and of determinant class, and 
$$T^{(2)}(T,\phi \circ i, \gamma \circ i)(t) = 1.$$

Finally, since $C^{(2)}_*(M,\phi \circ Q, \gamma \circ Q,t)$ is assumed weakly acyclic and of determinant class, it follows from Proposition \ref{prop mayer vietoris} that
$C_*^{(2)}(N,\phi,\gamma)(t)$ is weakly acyclic and of determinant class, and
$$
T^{(2)}(N,\phi,\gamma)(t)  \ \dot{=} \  
\dfrac{T^{(2)}(M,\phi \circ Q,\gamma \circ Q)(t)}
{\max(1,t)^{|\phi(\iota(c))|}}.
$$
\end{proof}

Let $M$ be the exterior of an oriented link $L= L_1 \cup \ldots \cup L_c$ in $S^3$ and $T = \partial V(L_c)$. 
Let $(\mu,\lambda)$ be a preferred meridian-longitude pair for $T$. 
We describe a simple closed curve $J$ on $T$ by its homotopy class, which is characterised by two relatively prime integers $p, q$:
$[J] = p [\mu] + q [\lambda]$.
Let $r,s \in \Z$ be relatively prime integers such that
$\det \left ( \begin{pmatrix} p & q \\ r & s \end{pmatrix} \right ) = 1$
and let $J'$ be a curve in $T$ such that 
$[J'] = r [\mu] + s [\lambda].$
 We can assume that $J$ and $J'$ intersect on a single point $P$. 
 Let $N$ denote the manifold obtained by Dehn filling on $L_c$ with coefficient $p/q$, and $B$ the filling solid torus.
 Then $Q: \pi_1(M) \twoheadrightarrow \pi_1(N)$ is the quotient group homomorphism that adds the relation $[\mu]^p [\lambda]^q = 1$. We have trivialised the curve $[J]$.
  Proposition \ref{prop surgery} can thus be re-written as:
  \begin{proposition}[\cite{BAthesis}, Theorem 3.7] \label{prop surgery links}
Let $\phi: \pi_N \to \Z$ and $\gamma: \pi_N \to G$ be group homomorphisms such that $(\pi_N,\phi,\gamma)$ forms an admissible triple. For all $t>0$, if $(\gamma \circ Q)([\mu]^r [\lambda]^s)$ is of infinite order in $G$ and if
$C_*^{(2)}(M,\phi \circ Q,\gamma \circ Q)(t)$ is weakly acyclic and of determinant class, then 
$C_*^{(2)}(N,\phi,\gamma)(t)$ has the same properties and
$$
T^{(2)}(N,\phi,\gamma)(t) \ \dot{=} \  
\dfrac{T^{(2)}(M,\phi \circ Q,\gamma \circ Q)(t)}
{\max(1,t)^{| r (\phi\circ Q)([\mu]) + s (\phi\circ Q)([\lambda]) |}}.
$$
\end{proposition}
Let us now study some applications of this Dehn surgery formula.

\subsection{$\infty$-surgery: erasing one component of a link}

Let $L = L_1 \cup \ldots \cup L_{c-1} \cup L_c$ be a $c$-component link, and $L' = L_1 \cup \ldots \cup L_{c-1}$ be the link obtained by forgetting the last component, or alternatively by applying a trivial Dehn filling of the last component.
Then the natural injection $i : M_L \hookrightarrow M_{L'}$  passes to fundamental groups as an epimorphism
$Q = i_*: G_L \twoheadrightarrow G_{L'}$, which is the same as the quotient homomorphism by the normal subgroup generated by any meridian of $L_c$. Let $(\mu_c,\lambda_c)$ be a preferred meridian-longitude system of $L_c$.
The surgery coefficients are $(p,q,r,s)=(1,0,0,1)$.
We can now state the Main Theorem of this article:

\begin{theorem} \label{thm surgery forget}
Let $\phi: \pi_1(M_{L'}) \to \Z$ and $\gamma: \pi_1(M_{L'}) \to G$ be group homomorphisms such that $(\pi_1(M_{L'}),\phi,\gamma)$ forms an admissible triple.
We can write $\phi = (n_1, \ldots, n_{c-1}) \circ \alpha_{L'}$ and thus 
$\phi \circ Q = (n_1, \ldots, n_{c-1}, 0) \circ \alpha_L$ for some non zero vector $(n_1, \ldots, n_{c-1}) \in \Z^{c-1}$.
 For all $t>0$, if $(\gamma \circ Q)([\lambda])$ is of infinite order in $G$ and if
 $C_*^{(2)}(M_L,(n_1, \ldots, n_{c-1}, 0) \circ \alpha_L,\gamma \circ Q)(t)$ is weakly acyclic and of determinant class, then 
$C_*^{(2)}(M_{L'},(n_1, \ldots, n_{c-1}) \circ \alpha_{L'},\gamma)(t)$ has the same properties and
$$T^{(2)}_{L',(n_1, \ldots, n_{c-1})}(\gamma)(t)  \ \dot{=} \ 
\dfrac{T^{(2)}_{L,(n_1, \ldots, n_{c-1}, 0)}(\gamma \circ Q)(t)}
{\max(1,t)^{|\mathrm{lk}(L_1,L_c) n_1 + \ldots + \mathrm{lk}(L_{c-1},L_c) n_{c-1}|}}.$$
\end{theorem}

This theorem generalises the well-known property of the Alexander polynomial for links proved by Torres in \cite{Torres}.

\begin{proof}
We apply Proposition \ref{prop surgery links} and we use the fact that here
\begin{align*}
 r (\phi\circ Q)([\mu]) + s (\phi\circ Q)([\lambda]) &= (\phi\circ Q)([\lambda_c]) \\ &= \left ((n_1, \ldots, n_{c-1}, 0) \circ \alpha_L\right ) ([\lambda_c]) \\
 &= \mathrm{lk}(L_1,L_c) n_1 + \ldots + \mathrm{lk}(L_{c-1},L_c) n_{c-1}.
\end{align*}
\end{proof}

\subsection{$1/n$-surgery: Twist knots and the Whitehead link}

Let $L$ be the Whitehead link in $S^3$, and $M_L$ its exterior. We draw it as in Figure \ref{fig whitehead link} with components $L_1$ and $L_2$. Note that $L$ is actually ambient isotopic to the link obtained by reordering the components, therefore doing a given surgery on $L_1$ or $L_2$ yields the same manifold up to homeomorphism. We will do a $1/n$-surgery on the component $L_2$.

\begin{figure}[h]
\centering
\begin{tikzpicture}[every path/.style={string ,black} , every node/.style={transform shape , knot crossing , inner sep=1.5 pt } ]
\begin{scope}[scale=0.8]
\begin{scope}[xshift=0cm,rotate=0,scale=1]
\draw  (4,2) node {$L_1$} ;
\coordinate (a) at (-0.37,1) ;
\coordinate (b) at (0.37,1) ;
\coordinate (c) at (-0.37,2) ;
\coordinate (d) at (0.37,2) ;
\coordinate (t1) at (-3,2) ;
\coordinate (t2) at (3,2) ;
\coordinate (t3) at (-3,1) ;
\coordinate (t4) at (3,1) ;
\coordinate (t5) at (-3,-1) ;
\coordinate (t6) at (3,-1) ;
\coordinate (t7) at (-3,-2) ;
\coordinate (t8) at (3,-2) ;
\draw (t1) -- (c);
\draw [->] (d) -- (t2) ..controls +(2.5,0) and +(2.5,0).. (t8) -- (2,-2);
\draw (2,-2) -- (-0.2,-2);
\draw  (-0.4,-2) -- (-2.2,-2); 
\draw (-2.4,-2) -- (t7) ..controls +(-2.5,0) and +(-2.5,0).. (t1) ;
\draw (t3) -- (a);
\draw (b) -- (t4) ..controls +(1,0) and +(1,0).. (t6) -- (2,-1);
\draw [<-] (2,-1) -- (-0.2,-1);
\draw  (-0.4,-1) -- (-2.2,-1); 
\draw (-2.4,-1) -- (t5) ..controls +(-1,0) and +(-1,0).. (t3) ;
\draw [very thick, color=red] (2.5,-1.3) -- (2.5,-1.1);
\draw[->] [very thick, color=red](2.5,-0.9) -- (2.5,-0.5);
\draw [very thick, color=red] (2.9,-0.6) node {$a_1$} ;
\draw [very thick, color=red](2.5,-1.7) -- (2.5,-1.9);
\draw[->] [very thick, color=red](2.5,-2.1) -- (2.5,-2.5);
\draw [very thick, color=red] (2.9,-2.4) node {$a_2$} ;
\begin{scope}[xshift=-2cm,rotate=0,scale=1]
\coordinate (yh) at (0,-0.7) ;
\coordinate (yb) at (0,-2.3) ;
\draw [very thick, color=blue] (0.3,-0.9) -- (0.3,-0.7) -- (0,-0.7);
\draw [very thick, color=blue] (0,-0.7) -- (-0.3,-0.7) -- (-0.3,-1.5);
\draw [very thick, color=blue][<-] (-0.3,-1.5) -- (-0.3,-2.3) -- (0.3,-2.3) -- (0.3,-2.1) ;
\draw [very thick, color=blue] (0.3,-1.1) -- (0.3,-1.9) ;
\draw [very thick, color=blue] (-0.5,-1.5) node {$\beta$} ;
\end{scope}
\begin{scope}[xshift=0cm,rotate=0,scale=1]
\coordinate (yh) at (0,-0.7) ;
\coordinate (yb) at (0,-2.3) ;
\draw (0.3,-0.9) -- (0.3,-0.5);
\draw (0.3,-0.3) -- (0.3,0) -- (-0.3,0) -- (-0.3,-1.5);
\draw [<-] (-0.3,-1.5) -- (-0.3,-3) -- (0.3,-3) -- (0.3,-2.1) ;
\draw  (0.3,-1.1) -- (0.3,-1.9) ;
\draw  (-0.65,-1.5) node {$L_2$} ;
\end{scope}
\coordinate (th) at (0,3) ;
\coordinate (td) at (5.5,0) ;
\coordinate (tb) at (0,-3) ;
\coordinate (tg) at (-5.5,0) ;
\draw [color=magenta, very thick] (2,0.4) -- (2,-0.4) -- (-0.2,-0.4);
\draw [->] [color=magenta, very thick] (-0.4,-0.4) -- (-1.5,-0.4) ;
\draw [color=magenta, very thick] (-1.5,-0.4) -- (-2,-0.4) -- (-2,0.4) -- (2,0.4) ;
\draw [color=magenta](-1.5,-0.15) node {$\alpha$} ;
\end{scope}
\begin{scope}[xshift=0cm,yshift=1.5cm,scale=0.63,color=black]
	\node (o) at (0,0) {};
	\node (h) at (90:0.4) {};	
	\node (hd) at (55:1) {};	
	\node (hdm) at (70:0.6) {};	
	\node (hg) at (180-55:1) {};
	\node (hgm) at (180-70:0.6) {};
	\node (b) at (-90:0.4) {};
	\node (bd) at (-55:1) {};	
	\node (bdm) at (-70:0.6) {};	
	\node (bg) at (-180+55:1) {};
	\node (bgm) at (-180+70:0.6) {};
	\draw (hd.center) .. controls (hd.2 south west) and (hdm.2 north east) .. (hdm.center) ;	
	\draw[<-] (hdm.center) .. controls (hdm.2 south west) and (h.2 north east) .. (h) ;	
	\draw (b.center) .. controls (b.4 north west) and (h.4 south west) .. (h) ;	
	\draw (b.center) .. controls (b.2 south east) and (bdm.2 north west) .. (bdm.center) ;	
	\draw[<-] (bdm.center) .. controls (bdm.2 south east) and (bd.2 north west) .. (bd.center) ;	
	\draw[->] (hg.center) .. controls (hg.2 south east) and (hgm.2 north west) .. (hgm.center) ;	
	\draw (hgm.center) .. controls (hgm.2 south east) and (h.2 north west) .. (h.center) ;	
	\draw (b) .. controls (b.4 north east) and (h.4 south east) .. (h.center) ;	
	\draw[->] (b) .. controls (b.2 south west) and (bgm.2 north east) .. (bgm.center) ;	
	\draw (bgm.center) .. controls (bgm.2 south west) and (bg.2 north east) .. (bg.center) ;	
\end{scope}
\end{scope}
\end{tikzpicture}
\caption{The Whitehead link} \label{fig whitehead link}
\end{figure}
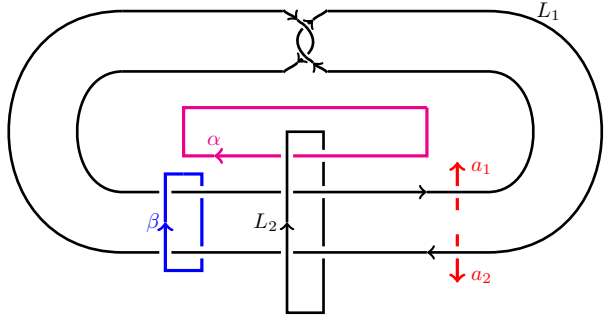

The following theorem relates a particular $L^2$-Alexander torsion of the Whitehead link, where $\phi$ sends the second component to zero and $\gamma$ is an epimorphism to a knot group, to the $L^2$-Alexander torsion of this knot group. The possible knots in question are the \textit{twist knots} $K_n$, described by the diagram of Figure \ref{fig twist knot}. 

\begin{figure}[h]
\centering
\begin{tikzpicture}[every path/.style={string ,black} , every node/.style={transform shape , knot crossing , inner sep=1.5 pt } ]

\begin{scope}[scale=0.8]
\begin{scope}[xshift=0cm,rotate=0,scale=1]
\draw  (4,2) node {$K_n$} ;
\coordinate (a) at (-0.37,1) ;
\coordinate (b) at (0.37,1) ;
\coordinate (c) at (-0.37,2) ;
\coordinate (d) at (0.37,2) ;
\coordinate (t1) at (-3,2) ;
\coordinate (t2) at (3,2) ;
\coordinate (t3) at (-3,1) ;
\coordinate (t4) at (3,1) ;
\coordinate (t5) at (-3,-1) ;
\coordinate (t6) at (3,-1) ;
\coordinate (t7) at (-3,-2) ;
\coordinate (t8) at (3,-2) ;
\draw (t1) -- (c);
\draw [->] (d) -- (t2) ..controls +(2.5,0) and +(2.5,0).. (t8) -- (2,-2);
\draw (2,-2) -- (1.8,-2);
\draw  (-1.4,-2) -- (-2.2,-2); 
\draw (-2.4,-2) -- (t7) ..controls +(-2.5,0) and +(-2.5,0).. (t1) ;
\draw (t3) -- (a);
\draw (b) -- (t4) ..controls +(1,0) and +(1,0).. (t6) -- (2,-1);
\draw [<-] (2,-1) -- (1.8,-1);
\draw  (-1.4,-1) -- (-2.2,-1); 
\draw (-2.4,-1) -- (t5) ..controls +(-1,0) and +(-1,0).. (t3) ;
\draw [very thick, color=red] (2.5,-1.3) -- (2.5,-1.1);
\draw[->] [very thick, color=red](2.5,-0.9) -- (2.5,-0.5);
\draw [very thick, color=red] (2.9,-0.6) node {$a_1$} ;
\draw [very thick, color=red](2.5,-1.7) -- (2.5,-1.9);
\draw[->] [very thick, color=red](2.5,-2.1) -- (2.5,-2.5);
\draw [very thick, color=red] (2.9,-2.4) node {$a_2$} ;
\begin{scope}[xshift=-2cm,rotate=0,scale=1]
\coordinate (yh) at (0,-0.7) ;
\coordinate (yb) at (0,-2.3) ;
\draw [very thick, color=blue] (0.3,-0.9) -- (0.3,-0.7) -- (0,-0.7);
\draw [very thick, color=blue] (0,-0.7) -- (-0.3,-0.7) -- (-0.3,-1.5);
\draw [very thick, color=blue][<-] (-0.3,-1.5) -- (-0.3,-2.3) -- (0.3,-2.3) -- (0.3,-2.1) ;
\draw [very thick, color=blue] (0.3,-1.1) -- (0.3,-1.9) ;
\draw [very thick, color=blue] (-0.5,-1.5) node {$m$} ;
\end{scope}
\coordinate (th) at (0,3) ;
\coordinate (td) at (5.5,0) ;
\coordinate (tb) at (0,-3) ;
\coordinate (tg) at (-5.5,0) ;
\draw  (0.2,-1.5) node {$\ldots$} ;
\draw  (0.2,-2.4) node {$\underbrace{\hspace{3cm}}$} ;
\draw  (0.2,-2.8) node {$2n$ crossings} ;
	\node (a1) at (-1.4,-1) {};
	\node (a2) at (-0.8,-1) {};
	\node (a3) at (-0.2,-1) {};
	\node (a4) at (0.6,-1) {};
	\node (a5) at (1.2,-1) {};
	\node (a6) at (1.8,-1) {};
	\node (b1) at (-1.1,-1.5) {};
	\node (b2) at (-0.5,-1.5) {};
	\node (b4) at (0.9,-1.5) {};
	\node (b5) at (1.5,-1.5) {};
	\node (c1) at (-1.4,-2) {};
	\node (c2) at (-0.8,-2) {};
	\node (c3) at (-0.2,-2) {};
	\node (c4) at (0.6,-2) {};
	\node (c5) at (1.2,-2) {};
	\node (c6) at (1.8,-2) {};
	\draw (c1.center) .. controls (c1.2 north east) and (b1.2 south west) .. (b1.center) ;	
	\draw (b1.center) .. controls (b1.2 north east) and (a2.2 south west) .. (a2.center) ;	
	\draw (c2.center) .. controls (c2.2 north east) and (b2.2 south west) .. (b2.center) ;	
	\draw (b2.center) .. controls (b2.2 north east) and (a3.2 south west) .. (a3.center) ;	
	\draw (c4.center) .. controls (c4.2 north east) and (b4.2 south west) .. (b4.center) ;	
	\draw (b4.center) .. controls (b4.2 north east) and (a5.2 south west) .. (a5.center) ;	
	\draw (c5.center) .. controls (c5.2 north east) and (b5.2 south west) .. (b5.center) ;	
	\draw (b5.center) .. controls (b5.2 north east) and (a6.2 south west) .. (a6.center) ;	
	\draw (a1.center) .. controls (a1.2 south east) and (b1.2 north west) .. (b1) ;	
	\draw (b1) .. controls (b1.2 south east) and (c2.2 north west) .. (c2.center) ;
	\draw (a2.center) .. controls (a2.2 south east) and (b2.2 north west) .. (b2) ;	
	\draw (b2) .. controls (b2.2 south east) and (c3.2 north west) .. (c3.center) ;
	\draw (a4.center) .. controls (a4.2 south east) and (b4.2 north west) .. (b4) ;	
	\draw (b4) .. controls (b4.2 south east) and (c5.2 north west) .. (c5.center) ;
	\draw (a5.center) .. controls (a5.2 south east) and (b5.2 north west) .. (b5) ;	
	\draw (b5) .. controls (b5.2 south east) and (c6.2 north west) .. (c6.center) ;
\end{scope}
\begin{scope}[xshift=0cm,yshift=1.5cm,scale=0.63,color=black]
	\node (o) at (0,0) {};
	\node (h) at (90:0.4) {};	
	\node (hd) at (55:1) {};	
	\node (hdm) at (70:0.6) {};	
	\node (hg) at (180-55:1) {};
	\node (hgm) at (180-70:0.6) {};
	\node (b) at (-90:0.4) {};
	\node (bd) at (-55:1) {};	
	\node (bdm) at (-70:0.6) {};	
	\node (bg) at (-180+55:1) {};
	\node (bgm) at (-180+70:0.6) {};
	\draw (hd.center) .. controls (hd.2 south west) and (hdm.2 north east) .. (hdm.center) ;	
	\draw[<-] (hdm.center) .. controls (hdm.2 south west) and (h.2 north east) .. (h) ;	
	\draw (b.center) .. controls (b.4 north west) and (h.4 south west) .. (h) ;	
	\draw (b.center) .. controls (b.2 south east) and (bdm.2 north west) .. (bdm.center) ;	
	\draw[<-] (bdm.center) .. controls (bdm.2 south east) and (bd.2 north west) .. (bd.center) ;	
	\draw[->] (hg.center) .. controls (hg.2 south east) and (hgm.2 north west) .. (hgm.center) ;	
	\draw (hgm.center) .. controls (hgm.2 south east) and (h.2 north west) .. (h.center) ;	
	\draw (b) .. controls (b.4 north east) and (h.4 south east) .. (h.center) ;	
	\draw[->] (b) .. controls (b.2 south west) and (bgm.2 north east) .. (bgm.center) ;	
	\draw (bgm.center) .. controls (bgm.2 south west) and (bg.2 north east) .. (bg.center) ;	
\end{scope}
\end{scope}
\end{tikzpicture}
\caption{The twist knot $K_n$} \label{fig twist knot}
\end{figure}
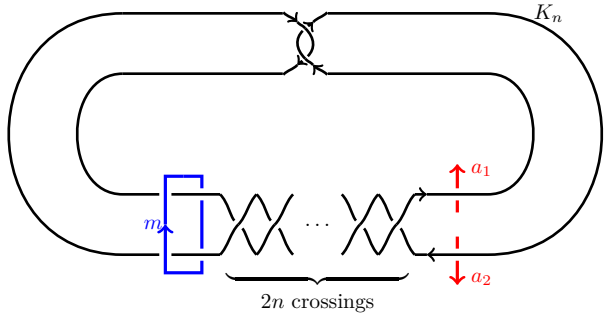

Note that $n\in \Z$ can be positive or negative, that $K_0 = O$ is the unknot, 
$K_1 = 3_1$ is the trefoil knot,
$K_{-1} = 4_1$ is the figure-eight knot,
$K_2 = 5_2$, $K_{-2} = 6_1$, etc.

Let $(\alpha, \beta)$ be a preferred meridian-longitude system for $L_2$ as in Figure \ref{fig whitehead link}. 
Here we do $1/n$-surgery on $L_2$, which means that $(p,q)=(1,n)$, and thus $(r,s) = (0,1)$ is a possible choice of coefficients for the curve $J'$, which means we can assume $J'= \beta$. 
Let $N$ be the manifold obtained by this surgery on $M_L$. Let $J_M: M_L \hookrightarrow N$ be the associated natural inclusion, which extends to an inclusion $S^3 \setminus V(L_2) \hookrightarrow S^3$ since $1/n$-surgery on the trivial knot in $S^3$ yields $S^3$. The image of $L_1$ by this inclusion is $K_n$, as Figures \ref{fig whitehead link} and \ref{fig twist knot} illustrate. Thus $N = M_{K_n} = S^3 \setminus V(K_n)$. 
The inclusion $J_M$ induces an epimorphism $Q_n: \pi_1(M_L) \twoheadrightarrow \pi_1(M_{K_n})$ whose kernel is the normal subgroup generated by $[J]= [\alpha] [\beta]^n$.
Thus the following diagram is commutative, and we obtain the following theorem:

\begin{center}
\begin{tikzpicture}
\begin{scope}[xshift=0cm,rotate=0,scale=1]
[description/.style={fill=white,inner sep=2pt}] 
\matrix(a)[matrix of math nodes, row sep=2em, column sep=2.5em, text height=1.5ex, text depth=0.25ex] 
{ \pi_1(M_L) &  \pi_1(M_{K_n}) & G \\
\Z^2 & \Z & \\}; 
\path[->](a-1-1) edge node[above]{$Q_n$} (a-1-2); 
\path[->](a-1-2) edge node[above]{$\gamma$} (a-1-3); 
\path[->](a-1-1) edge node[right]{$\alpha_L$} (a-2-1); 
\path[->](a-1-2) edge node[right]{$\alpha_{K_n}$} (a-2-2); 
\path[->](a-2-1) edge node[below]{$(1,0)$} (a-2-2); 
\path[->, dotted](a-1-3) edge  (a-2-2);
\end{scope}
\end{tikzpicture}
\end{center}

\begin{theorem} \label{thm surgery twist}
Let $\gamma: \pi_1(M_{K_n}) \to G$ be a group homomorphism such that \\
$(\pi_1(M_{K_n}),\alpha_{K_n},\gamma)$ forms an admissible triple.
 For all $t>0$, if $\gamma (m)$ is of infinite order in $G$ and if
$C_*^{(2)}(M_L,(1,0) \circ \alpha_L,\gamma \circ Q_n)(t)$ is weakly acyclic and of determinant class, then
$C_*^{(2)}(M_{K_n},\alpha_{K_n},\gamma)(t)$ has the same properties, and
$$
T^{(2)}_{K_n,1}(\gamma)(t) \ \dot{=} \ 
T^{(2)}_{L,(1,0)}(\gamma \circ Q_n)(t).
$$
In particular, for $\gamma=id$,
$$
\Delta^{(2)}_{K_n}(t) \ \dot{=} \ 
T^{(2)}_{L,(1,0)}(Q_n)(t) \cdot \max(1,t).
$$
\end{theorem}

\begin{proof}
We apply Proposition \ref{prop surgery links}. 
Here $(\gamma \circ Q_n)([\alpha]^r [\beta]^s) = (\gamma \circ Q_n)([\beta]) = \gamma(m)$.
Since
$(\alpha_{K_n} \circ Q_n)([\beta]) = (1,0) \cdot \alpha_L([\beta]) = (1,0) \cdot \begin{pmatrix}0 \\ 0\end{pmatrix} = 0$, we conclude that
 the denominator part is equal to $\max(1,t)^{(\alpha_{K_n} \circ Q_n)([\beta])} = 1$. The theorem follows. 
\end{proof}

\begin{remark}
Results of Thurston and Jorgensen (see \cite{thurston_notes}) demonstrate that if one does $p/q$-Dehn filling on a hyperbolic link complement, with $p^2 + q^2$ large enough the resulting manifold will also be hyperbolic with volume approaching the volume of the original link complement by smaller values as $p^2+q^2 \to \infty$.
In particular, as $n \to \infty$, by Theorem \ref{thm surgery twist},
\begin{align*}
 T^{(2)}_{L,(1,0)}(Q_n)(1) = 
\Delta^{(2)}_{K_n}(1) 
&= \exp\left (\dfrac{vol(K_n)}{6 \pi}\right ) 
&\underset{n \to \infty}{\longrightarrow}
\exp\left (\dfrac{vol(L)}{6 \pi}\right )  = T^{(2)}_{L,(1,0)}(1).
\end{align*}
\end{remark}
It is now natural to wonder if there exists a similar convergence of the $L^2$-Alexander torsions for $t \neq 1$.

\begin{question}
For every $t>0$, do we have
\begin{align*}
 T^{(2)}_{L,(1,0)}(Q_n)(t) 
\underset{n \to \infty}{\longrightarrow}
T^{(2)}_{L,(1,0)}(t) \ ?
\end{align*}
\end{question}

\section{Seifert-fibered link exteriors}
\label{sec:splicing}

It follows from \cite[Theorem 8.5]{DFL} and \cite{herrmannarxiv}
that for a Seifert $3$-manifold $M$, the $L^2$-Alexander torsions $T^{(2)}(M,\phi,\gamma)$ are equal to
$(t \mapsto \max(1,t))$ to the power the Thurston norm $x_M(\phi)$ of $\phi$.
Computing Thurston norms is a difficult problem in general, but
 we will compute in this section the exact values of the $L^2$-Alexander torsions for all  Seifert-fibered link exteriors, and thus the values of their associated Thurston norms. 
We hope that these various formulas will provide help to the community to compute particular Thurston norms for link exteriors.

Along the way, these various computations allow us to determine the $L^2$-Alexander torsions of a connected sum of links and of a general multi-component cabling of a link by a torus link (see Theorems \ref{thm_torsion_sum} and \ref{thm_torsion_cabling}).
These results generalise the ones for knots of \cite[Theorem 3.2]{BA13} and \cite[Theorem 4.3]{BA13}.

\subsection{Links with Seifert-fibered exterior} \label{Links with Seifert-fibered exterior}

Let us consider $S^3$ both as the unit sphere of $\C^2$ and as the one-point compactification of $\R^3$ by the point $\infty$. We define
\begin{itemize}
\item $T(m,n) = \{(z_1,z_2) \in S^3 \subset \C^2 | z_1^m = z_2^n\}$ the torus link of type $(m,n)$ with
 $e = \textrm{gcd}(m,n)$ components (which can be drawn on a torus as $m$ strands twisted $n$ times by an angle $2\pi / m$),
\item $H_v =  \{(z_1,0) \in S^3 \}$ the trivial knot drawn as the vertical line passing through $\infty$ in $\R^3$,
\item $H_h =  \{(0,z_2) \in S^3 \}$ the trivial knot drawn in $\R^3$ as the unit circle of an horizontal plane (normal to $H_v$ in its origin).
\end{itemize}
This allows us to describe the links $L$ in $S^3$ whose exterior is a Seifert manifold \cite[Proposition 3.3]{Bud}:

\begin{proposition} \label{prop link seifert}
Let $L$ be a non-split link in $S^3$. Its exterior $M_L$ is Seifert-fibered if and only if $L$ is one of the following links:
\begin{itemize}
\item a torus link $T(m,n) = T(ep,eq)$ with $p,q$ relatively prime (and both nonzero if $e \geqslant 2$),
\item a link $T(ep,eq) \cup H_v$ with $p,q$ relatively prime and $p \neq 0$,
\item a link $T(ep,eq) \cup H_v \cup H_h$ with $p,q$ relatively prime.
\end{itemize}
\end{proposition}

We exclude the torus links of the form $T(m,0)$ with $|m| \geqslant 2$ since they are split.
We want to compute the $L^2$-Alexander torsions of all links listed in Proposition \ref{prop link seifert}. We will need various tools for this: gluing formulas, explicit homeomorphisms between link exteriors, $\infty$-surgery, etc. For the reader's convenience we outline the several steps of our strategy:
\begin{enumerate}
\item We compute the torsions for the keychain links $T(e,0) \cup H_v$ with Proposition \ref{prop WS1} (Section \ref{sec:key}).
\item We deduce the torsions for a connected sum of links thanks to the gluing formula of Proposition \ref{prop mayer vietoris} (Section \ref{Connected sum for links}).
\item We compute the torsions for the links $T(e,ek) \cup H_v$ by identifying their exterior with the exterior of the keychain link $T(e,0) \cup H_v$ (Section \ref{e,ek inside torus}).
\item We compute the torsions for the links $T(p,q) \cup H_v \cup H_h$ with the gluing formula (Section \ref{sec:tpqhh}).
\item We deduce the torsions for the links $T(ep,eq) \cup H_v \cup H_h$ thanks to the gluing formula (Section \ref{sec:tepeqhh}).
\item We apply two successive $\infty$-surgeries with Theorem \ref{thm surgery forget} and deduce the torsions for the links 
$T(ep,eq) \cup H_v$ and $T(ep,eq)$ (Sections \ref{sec:tepeqh} and \ref{sec:tepeq}).
\item We deduce general cabling formulas for links, thanks to the gluing formula (Section \ref{sec:cabling}).
\end{enumerate}

\begin{remark}
Most of the following results can also be proven with Fox calculus on particular presentations of the link groups. Details can be found in \cite{BAthesis}.
\end{remark}

\subsection{Keychain links}
\label{sec:key}

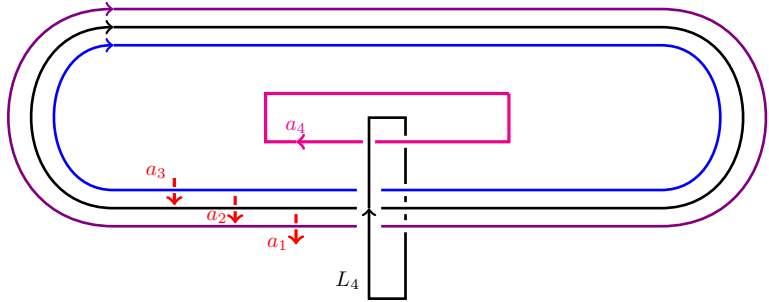
\begin{figure}[!h]
\centering
\begin{tikzpicture}[every path/.style={string ,black} , every node/.style={transform shape , knot crossing , inner sep=1.5 pt } ]
\begin{scope}[scale=0.8]

\begin{scope}[xshift=0cm,rotate=0,scale=1]

\coordinate (ch) at (-2,-1) ;
\coordinate (cb) at (-2,-2) ;

\coordinate (t1) at (-3,2) ;
\coordinate (t2) at (3,2) ;
\coordinate (t3) at (-3,1) ;
\coordinate (t4) at (3,1) ;
\coordinate (t5) at (-3,-1) ;
\coordinate (t6) at (3,-1) ;
\coordinate (t7) at (-3,-2) ;
\coordinate (t8) at (3,-2) ;

\coordinate (k1) at (-4.5,1.8) ;
\coordinate (k2) at (4.5,1.8) ;
\coordinate (k3) at (-4.5,1.2) ;
\coordinate (k4) at (4.5,1.2) ;
\coordinate (k5) at (-4.5,-1.2) ;
\coordinate (k6) at (4.5,-1.2) ;
\coordinate (k7) at (-4.5,-1.8) ;
\coordinate (k8) at (4.5,-1.8) ;
\coordinate (k9) at (-4.5,1.5) ;
\coordinate (k10) at (4.5,1.5) ;
\coordinate (k11) at (4.5,-1.5) ;
\coordinate (k12) at (-4.5,-1.5) ;
\draw[color=violet] (k1) -- (k2) ..controls +(2.3,0) and +(2.3,0).. (k8) -- (-0.1,-1.8);
\draw[->][color=violet] (-0.5,-1.8) -- (k7) ..controls +(-2.3,0) and +(-2.3,0).. (k1) ;
\draw[color=blue] (k3) -- (k4) ..controls +(1.3,0) and +(1.3,0).. (k6) -- (-0.1,-1.2);
\draw[->][color=blue] (-0.5,-1.2) -- (k5) ..controls +(-1.3,0) and +(-1.3,0).. (k3) ;
\draw  (k9) -- (k10) ..controls +(1.8,0) and +(1.8,0).. (k11) -- (-0.1,-1.5);
 \draw[->]  (-0.5,-1.5) -- (k12) ..controls +(-1.8,0) and +(-1.8,0).. (k9) ;

\begin{scope}[xshift=-4cm, yshift=0cm]
\draw [very thick, color=red](2.5,-1.6) -- (2.5,-1.75);
\draw[->] [very thick, color=red](2.5,-1.85) -- (2.5,-2.1);
\draw [very thick, color=red] (2.2,-2.1) node {$a_1$} ;
\end{scope}

\begin{scope}[xshift=-5cm, yshift=0.3cm]
\draw [very thick, color=red](2.5,-1.6) -- (2.5,-1.75);
\draw[->] [very thick, color=red](2.5,-1.85) -- (2.5,-2.05);
\draw [very thick, color=red] (2.2,-1.95) node {$a_2$} ;
\end{scope}

\begin{scope}[xshift=-6cm, yshift=0.6cm]
\draw [very thick, color=red](2.5,-1.6) -- (2.5,-1.75);
\draw[->] [very thick, color=red](2.5,-1.85) -- (2.5,-2.05);
\draw [very thick, color=red] (2.2,-1.5) node {$a_3$} ;
\end{scope}

\end{scope}

\begin{scope}[xshift=0cm,rotate=0,scale=1]
\coordinate (yh) at (0,-0.7) ;
\coordinate (yb) at (0,-2.3) ;
\draw (0.3,-1.1) -- (0.3,-0.5);
\draw (0.3,-1.3) -- (0.3,-1.4);
\draw (0.3,-1.6) -- (0.3,-1.7);
\draw (0.3,-0.3) -- (0.3,0) -- (-0.3,0) -- (-0.3,-1.5);
\draw [<-] (-0.3,-1.5) -- (-0.3,-3) -- (0.3,-3) -- (0.3,-1.9) ;
\draw  (-0.65,-2.7) node {$L_4$} ;
\end{scope}

\coordinate (th) at (0,3) ;
\coordinate (td) at (5.5,0) ;
\coordinate (tb) at (0,-3) ;
\coordinate (tg) at (-5.5,0) ;

\draw [color=magenta, very thick] (2,0.4) -- (2,-0.4) -- (-0.2,-0.4);
\draw [->] [color=magenta, very thick] (-0.4,-0.4) -- (-1.5,-0.4) ;
\draw [color=magenta, very thick] (-1.5,-0.4) -- (-2,-0.4) -- (-2,0.4) -- (2,0.4) ;
\draw [color=magenta](-1.5,-0.15) node {$a_4$} ;

\end{scope}
\end{tikzpicture}
\caption{The keychain link, for $e=3$} \label{fig keychain 4}
\end{figure}

Let $e \geqslant 1$. Let $L$ be the $(e+1)$-component link $KC_{e+1} = T(e,0) \cup H_v$. We call it the \textit{$(e+1)$-keychain link}, see Figure \ref{fig keychain 4}. Let us call $L_1, \ldots L_e$ the $e$ parallel components of $T(e,0)$ and $L_{e+1} = H_v$ the one that circles them all.
The link exterior $M_L$ is homeomorphic to $S^1 \times W$ where $W$ is a disc with $e$ punctures. Thus the link group $G_L = \pi_1(M_L)$ is isomorphic to $\F[g_1, \ldots g_e] \times \Z$.
The abelianisation $\alpha_L: G_L \to \Z^{e+1}$ sends $a_i$, the meridian of $L_i$, to the $i$-th vector of the natural base of  $\Z^{e+1}$.

\begin{proposition} \label{torsion_keychain}
Let $e \geqslant 1$. 
 The $L^2$-Alexander torsion for the exterior of the\\
  $(e+1)$-component keychain link $L$  is non-zero for all admissible triples of the form $(G_L, (n_1 , \ldots,  n_{e+1}) \circ \alpha_L, \gamma)$ such that
$\gamma(a_{e+1})$ has infinite order in $G$
 and for all $t>0$. One has:
$$ T^{(2)}_{L, (n_1 , \ldots, n_{e+1})}(\gamma)(t)
\ \dot{=} \ \max(1,t)^{(e-1)|n_{e+1}|}.$$
\end{proposition}

\begin{proof}
We apply Proposition \ref{prop WS1} with $W$ a disc with $e$ punctures.
\end{proof}

\subsection{Connected sum for links} \label{Connected sum for links}

Let $L = L_1 \cup \ldots \cup L_{c+1}$ and  $L' = L'_1 \cup \ldots \cup L'_{d+1}$ be two non-split links in $S^3$ such that $L \cup L'$ is split. Let $L''$ be the $(c+d+1)$-component link obtained by
deleting small parts of $L_{c+1}$ and of $L'_{d+1}$
and then connecting them to form a single component denoted $L_{c+1}\sharp L'_{d+1}$ (in a way that respects the orientations of $L_{c+1}$ and of $L'_{d+1}$).
The link $L''$ is \textit{the connected sum of $L$ and $L'$ along the components $L_{c+1}$ and $L'_{d+1}$}, and we order its components in the following way:
\begin{align*}
L'' &= L''_1 \cup \ldots \cup L''_c \cup L''_{c+1} \cup \ldots \cup L''_{c+d} \cup L''_{c+d+1} \\
&= L_1 \cup \ldots \cup L_c \cup L'_{1} \cup \ldots \cup L'_{d} \cup (L_{c+1}\sharp L'_{d+1}).
\end{align*}
The manifold $M_{L''}$ is the toroidal gluing of $M_L$, $M_{L'}$ and a $3$-component keychain link $KC = T(2,0) \cup H_v = \mathcal{K}_1 \cup \mathcal{K}_2 \cup \mathcal{K}_3$, where $L_{c+1}$ is glued with $\mathcal{K}_1$, $L'_{d+1}$ is glued with $\mathcal{K}_2$, and the boundary of $\mathcal{K}_3$ becomes the boundary of $L''_{c+d+1}$. For details and examples we refer to \cite{Bud}.
Let $n_1, \ldots, n_{c+d+1} \in \Z$ and let $\gamma: G_{L''} \to G$ such that
 $(G_{L''}, (n_1, \ldots, n_{c+d+1}) \circ \alpha_{L''}, \gamma)$ is an admissible triple. Let $t>0$.
Let $J: M_{L} \hookrightarrow M_{L''}$ and $J':M_{L'} \hookrightarrow M_{L''}$ denote the inclusions associated with the toroidal gluing and $j,j'$ the induced injective group homomorphisms on the fundamental groups (see the following diagram for clarity).

\begin{tikzpicture}
[description/.style={fill=white,inner sep=2pt}] 
\matrix(a)[matrix of math nodes, row sep=2em, column sep=1.5em, text height=1.5ex, text depth=0.25ex] 
{ \partial(V(L_{c+1})) & M_L & \\  & M_{KC} & M_{L''} \\ \partial(V(L'_{d+1})) & M_{L'} & \\}; 
\path[->](a-1-1) edge node[below]{} (a-1-2); 
\path[->](a-1-1) edge node[below]{} (a-2-2); 
\path[->](a-3-1) edge node[above]{} (a-2-2);  
\path[->](a-3-1) edge node[above]{} (a-3-2);  
\path[->](a-1-2) edge node[above]{$J$} (a-2-3); 
\path[->](a-3-2) edge node[below]{$J'$} (a-2-3); 
\path[->](a-2-2) edge node[above]{} (a-2-3);  

\begin{scope}[xshift=5.4cm,rotate=0,scale=1]
[description/.style={fill=white,inner sep=2pt}] 
\matrix(a)[matrix of math nodes, row sep=2em, column sep=2em, text height=1.5ex, text depth=0.25ex] 
{ \Z^2 & G_L\\  & G_{KC} & G_{L''}  & G\\ \Z^2 & G_{L'}  & \Z^{c+d+1} & \Z\\}; 
\path[->](a-1-1) edge node[below]{} (a-1-2); 
\path[->](a-1-1) edge node[below]{} (a-2-2); 
\path[->](a-3-1) edge node[above]{} (a-2-2);  
\path[->](a-3-1) edge node[above]{} (a-3-2);  
\path[->](a-1-2) edge node[above]{$j$} (a-2-3); 
\path[->](a-3-2) edge node[below]{$j'$} (a-2-3); 
\path[->](a-2-2) edge node[above]{} (a-2-3);  
\path[->](a-2-3) edge node[above]{$\gamma$} (a-2-4);
\path[->](a-2-3) edge node[right]{$\alpha_{L''}$} (a-3-3);
\path[->](a-3-3) edge node[below]{$(n_1, \ldots, n_{c+d+1})$} (a-3-4);
\path[->, dotted](a-2-4) edge  (a-3-4);
\end{scope}
\end{tikzpicture}

We can see that 
$$(n_1, \ldots, n_{c+d+1}) \circ \alpha_{L''} \circ j = (n_1, \ldots, n_c, n_{c+d+1}) \circ \alpha_{L}$$
and that
$$(n_1, \ldots, n_{c+d+1}) \circ \alpha_{L''} \circ j' = (n_{c+1}, \ldots, n_{c+d}, n_{c+d+1}) \circ \alpha_{L'}$$
by checking these identities on each of the meridian curves of $L''$.
Let $m''_{c+d+1}$ denote a preferred meridian of $L''_{c+d+1}$. Then $m''_{c+d+1} = j(m_{c+1}) = j'(m'_{d+1})$ where
 $m_{c+1}$ is a preferred meridian of $L_{c+1}$ and
  $m'_{d+1}$ is a preferred meridian of $L'_{d+1}$.

\begin{theorem} \label{thm_torsion_sum}
Assume that \begin{itemize}
\item $C_*^{(2)}(M_L,(n_1, \ldots, n_c, n_{c+d+1}) \circ \alpha_{L}, \gamma \circ j,t)$ is weakly acyclic and of determinant class,
\item $C_*^{(2)}(M_{L'},(n_{c+1}, \ldots, n_{c+d}, n_{c+d+1}) \circ \alpha_{L'}, \gamma \circ j',t)$ is weakly acyclic and of determinant class,
\item $\gamma(m''_{c+d+1})$ is of infinite order in $G$,
\end{itemize}
then $C_*^{(2)}(M_{L''},(n_1, \ldots, n_{c+d+1}) \circ \alpha_{L''}, \gamma,t)$ is weakly acyclic and of determinant class, and
$$\dfrac{
T^{(2)}_{L'',(n_1, \ldots, n_{c+d+1})}(\gamma)(t) 
}{ \max(1,t)^{|n_{c+d+1}|}}
\ \dot{=} \ 
T^{(2)}_{L,(n_1, \ldots,n_c, n_{c+d+1})}(\gamma \circ j)(t)
\cdot
T^{(2)}_{L',(n_{c+1}, \ldots, n_{c+d}, n_{c+d+1})}(\gamma \circ j')(t)
.$$
\end{theorem}

This theorem generalizes the sum formula for the $L^2$-Alexander invariant of knots proven in
\cite[Theorem 3.2]{BA13} (in this case $c=0, d=0, n_1=1$).

\begin{proof}
We use Proposition \ref{prop t2 JSJ} and Theorem \ref{torsion_keychain}. Since we assume that $\gamma(m''_{c+d+1})$ is of infinite order in $G$, it follows that the tori $\partial(V(\mathcal{K}_1))$ and $\partial(V(\mathcal{K}_2))$ have infinite image under $\gamma$, because their preferred longitudes are homotopic to a preferred meridian of $\mathcal{K}_3$ which is sent to $m''_{c+d+1}$.
The formula follows then from Proposition \ref{prop t2 JSJ} and Theorem \ref{torsion_keychain}. 
\end{proof}

\subsection{The link $T(e,ek) \cup H_v$} \label{e,ek inside torus}

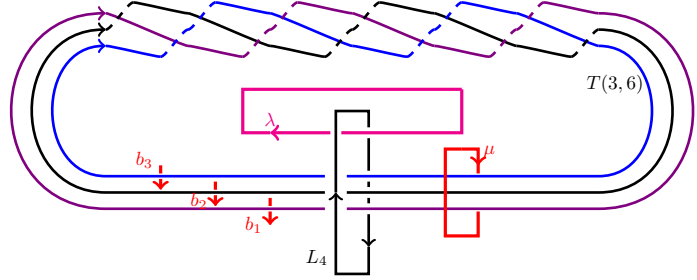
\begin{figure}[!h]
\centering
\begin{tikzpicture}[every path/.style={string ,black} , every node/.style={transform shape , knot crossing , inner sep=1.5 pt } ]

\begin{scope}[scale=0.8]

\begin{scope}[xshift=0cm,rotate=0,scale=0.9]

\begin{scope}[xshift=0cm,rotate=0,scale=1]

\coordinate (ch) at (-2,-1) ;
\coordinate (cb) at (-2,-2) ;
\coordinate (t1) at (-3,2) ;
\coordinate (t2) at (3,2) ;
\coordinate (t3) at (-3,1) ;
\coordinate (t4) at (3,1) ;
\coordinate (t5) at (-3,-1) ;
\coordinate (t6) at (3,-1) ;
\coordinate (t7) at (-3,-2) ;
\coordinate (t8) at (3,-2) ;
\coordinate (k1) at (-4.5,1.8) ;
\coordinate (k2) at (4.5,1.8) ;
\coordinate (k3) at (-4.5,1.2) ;
\coordinate (k4) at (4.5,1.2) ;
\coordinate (k5) at (-4.5,-1.2) ;
\coordinate (k6) at (4.5,-1.2) ;
\coordinate (k7) at (-4.5,-1.8) ;
\coordinate (k8) at (4.5,-1.8) ;
\coordinate (k9) at (-4.5,1.5) ;
\coordinate (k10) at (4.5,1.5) ;
\coordinate (k11) at (4.5,-1.5) ;
\coordinate (k12) at (-4.5,-1.5) ;
\draw[color=violet] (k2) ..controls +(2.3,0) and +(2.3,0).. (k8) -- (-0.1,-1.8);
\draw[->][color=violet] (-0.5,-1.8) -- (k7) ..controls +(-2.3,0) and +(-2.3,0).. (k1) ;
\draw[color=blue] (k4) ..controls +(1.3,0) and +(1.3,0).. (k6) -- (-0.1,-1.2);
\draw[->][color=blue] (-0.5,-1.2) -- (k5) ..controls +(-1.3,0) and +(-1.3,0).. (k3) ;
\draw  (k10) ..controls +(1.8,0) and +(1.8,0).. (k11) -- (-0.1,-1.5);
 \draw[->]  (-0.5,-1.5) -- (k12) ..controls +(-1.8,0) and +(-1.8,0).. (k9) ;

\begin{scope}[xshift=-4cm, yshift=0cm]
\draw [very thick, color=red](2.5,-1.6) -- (2.5,-1.75);
\draw[->] [very thick, color=red](2.5,-1.85) -- (2.5,-2.1);
\draw [very thick, color=red] (2.2,-2.1) node {$b_1$} ;
\end{scope}

\begin{scope}[xshift=-5cm, yshift=0.3cm]
\draw [very thick, color=red](2.5,-1.6) -- (2.5,-1.75);
\draw[->] [very thick, color=red](2.5,-1.85) -- (2.5,-2.05);
\draw [very thick, color=red] (2.2,-1.95) node {$b_2$} ;
\end{scope}

\begin{scope}[xshift=-6cm, yshift=0.6cm]
\draw [very thick, color=red](2.5,-1.6) -- (2.5,-1.75);
\draw[->] [very thick, color=red](2.5,-1.85) -- (2.5,-2.05);
\draw [very thick, color=red] (2.2,-1.5) node {$b_3$} ;
\end{scope}

\begin{scope}[xshift=-4.5cm]
\coordinate (kh) at (0,1.8) ;
\coordinate (km) at (0,1.5) ;
\coordinate (kb) at (0,1.2) ;
\coordinate (khi) at (0.5,2) ;
\coordinate (kbi) at (1,1) ;
\coordinate (kh') at (1.5,1.8) ;
\coordinate (km') at (1.5,1.5) ;
\coordinate (kb') at (1.5,1.2) ;
\draw  [color=violet](kh) ..controls +(0.1,0) and +(-0.1,0).. (kb') ;
\draw [style=dashed] (km) ..controls +(0.1,0) and +(-0.1,0).. (khi) ;
\draw  (khi) ..controls +(0.1,0) and +(-0.1,0).. (kh') ;
\draw  [color=blue](kb) ..controls +(0.1,0) and +(-0.1,0).. (kbi) ;
\draw [color=blue][style=dashed] (kbi) ..controls +(0.1,0) and +(-0.1,0).. (km') ;
\end{scope}

\begin{scope}[xshift=-3cm]
\coordinate (kh) at (0,1.8) ;
\coordinate (km) at (0,1.5) ;
\coordinate (kb) at (0,1.2) ;
\coordinate (khi) at (0.5,2) ;
\coordinate (kbi) at (1,1) ;
\coordinate (kh') at (1.5,1.8) ;
\coordinate (km') at (1.5,1.5) ;
\coordinate (kb') at (1.5,1.2) ;
\draw  (kh) ..controls +(0.1,0) and +(-0.1,0).. (kb') ;
\draw [color=blue][style=dashed] (km) ..controls +(0.1,0) and +(-0.1,0).. (khi) ;
\draw  [color=blue](khi) ..controls +(0.1,0) and +(-0.1,0).. (kh') ;
\draw  [color=violet](kb) ..controls +(0.1,0) and +(-0.1,0).. (kbi) ;
\draw [color=violet][style=dashed] (kbi) ..controls +(0.1,0) and +(-0.1,0).. (km') ;
\end{scope}

\begin{scope}[xshift=-1.5cm]
\coordinate (kh) at (0,1.8) ;
\coordinate (km) at (0,1.5) ;
\coordinate (kb) at (0,1.2) ;
\coordinate (khi) at (0.5,2) ;
\coordinate (kbi) at (1,1) ;
\coordinate (kh') at (1.5,1.8) ;
\coordinate (km') at (1.5,1.5) ;
\coordinate (kb') at (1.5,1.2) ;
\draw  [color=blue](kh) ..controls +(0.1,0) and +(-0.1,0).. (kb') ;
\draw [color=violet][style=dashed] (km) ..controls +(0.1,0) and +(-0.1,0).. (khi) ;
\draw  [color=violet](khi) ..controls +(0.1,0) and +(-0.1,0).. (kh') ;
\draw  (kb) ..controls +(0.1,0) and +(-0.1,0).. (kbi) ;
\draw [style=dashed] (kbi) ..controls +(0.1,0) and +(-0.1,0).. (km') ;
\end{scope}

\begin{scope}[xshift=0cm]
\coordinate (kh) at (0,1.8) ;
\coordinate (km) at (0,1.5) ;
\coordinate (kb) at (0,1.2) ;
\coordinate (khi) at (0.5,2) ;
\coordinate (kbi) at (1,1) ;
\coordinate (kh') at (1.5,1.8) ;
\coordinate (km') at (1.5,1.5) ;
\coordinate (kb') at (1.5,1.2) ;
\draw  [color=violet](kh) ..controls +(0.1,0) and +(-0.1,0).. (kb') ;
\draw [style=dashed] (km) ..controls +(0.1,0) and +(-0.1,0).. (khi) ;
\draw  (khi) ..controls +(0.1,0) and +(-0.1,0).. (kh') ;
\draw  [color=blue](kb) ..controls +(0.1,0) and +(-0.1,0).. (kbi) ;
\draw [color=blue][style=dashed] (kbi) ..controls +(0.1,0) and +(-0.1,0).. (km') ;
\end{scope}

\begin{scope}[xshift=1.5cm]
\coordinate (kh) at (0,1.8) ;
\coordinate (km) at (0,1.5) ;
\coordinate (kb) at (0,1.2) ;
\coordinate (khi) at (0.5,2) ;
\coordinate (kbi) at (1,1) ;
\coordinate (kh') at (1.5,1.8) ;
\coordinate (km') at (1.5,1.5) ;
\coordinate (kb') at (1.5,1.2) ;
\draw  (kh) ..controls +(0.1,0) and +(-0.1,0).. (kb') ;
\draw [color=blue][style=dashed] (km) ..controls +(0.1,0) and +(-0.1,0).. (khi) ;
\draw  [color=blue](khi) ..controls +(0.1,0) and +(-0.1,0).. (kh') ;
\draw  [color=violet](kb) ..controls +(0.1,0) and +(-0.1,0).. (kbi) ;
\draw [color=violet][style=dashed] (kbi) ..controls +(0.1,0) and +(-0.1,0).. (km') ;
\end{scope}

\begin{scope}[xshift=3cm]
\coordinate (kh) at (0,1.8) ;
\coordinate (km) at (0,1.5) ;
\coordinate (kb) at (0,1.2) ;
\coordinate (khi) at (0.5,2) ;
\coordinate (kbi) at (1,1) ;
\coordinate (kh') at (1.5,1.8) ;
\coordinate (km') at (1.5,1.5) ;
\coordinate (kb') at (1.5,1.2) ;
\draw  [color=blue](kh) ..controls +(0.1,0) and +(-0.1,0).. (kb') ;
\draw [color=violet][style=dashed] (km) ..controls +(0.1,0) and +(-0.1,0).. (khi) ;
\draw  [color=violet](khi) ..controls +(0.1,0) and +(-0.1,0).. (kh') ;
\draw  (kb) ..controls +(0.1,0) and +(-0.1,0).. (kbi) ;
\draw [style=dashed] (kbi) ..controls +(0.1,0) and +(-0.1,0).. (km') ;
\end{scope}

\begin{scope}[xshift=2cm,rotate=0,scale=1]
\coordinate (yh) at (0,-0.7) ;
\coordinate (yb) at (0,-2.3) ;
\draw [<-] [very thick, color=red] (0.3,-1) -- (0.3,-0.7) -- (0,-0.7);
\draw [very thick, color=red] (0.3,-0.9) -- (0.3,-1.15);
\draw [very thick, color=red] (0,-0.7) -- (-0.3,-0.7) -- (-0.3,-2.3) -- (0.3,-2.3) -- (0.3,-2) -- (0.3,-1.85) ;
\draw [very thick, color=red] (0.5,-0.75) node {$\mu$} ;
\end{scope}
\draw  (4.8,0.5) node {$T(3,6)$} ;
\end{scope}

\begin{scope}[xshift=0cm,rotate=0,scale=1]
\coordinate (yh) at (0,-0.7) ;
\coordinate (yb) at (0,-2.3) ;
\draw (0.3,-1.1) -- (0.3,-0.5);
\draw (0.3,-1.3) -- (0.3,-1.4);
\draw (0.3,-1.6) -- (0.3,-1.7);
\draw [->] (0.3,-1.9) -- (0.3,-2.5);
\draw (0.3,-0.3) -- (0.3,0) -- (-0.3,0) -- (-0.3,-1.5);
\draw [<-] (-0.3,-1.5) -- (-0.3,-3) -- (0.3,-3) -- (0.3,-2.5) ;
\draw  (-0.65,-2.7) node {$L_4$} ;
\end{scope}

\coordinate (th) at (0,3) ;
\coordinate (td) at (5.5,0) ;
\coordinate (tb) at (0,-3) ;
\coordinate (tg) at (-5.5,0) ;

\draw [color=magenta, very thick] (2,0.4) -- (2,-0.4) -- (-0.2,-0.4);
\draw [->] [color=magenta, very thick] (-0.4,-0.4) -- (-1.5,-0.4) ;
\draw [color=magenta, very thick] (-1.5,-0.4) -- (-2,-0.4) -- (-2,0.4) -- (2,0.4) ;
\draw [color=magenta](-1.5,-0.15) node {$\lambda$} ;

\end{scope}
\end{scope}
\end{tikzpicture}
\caption{The link $T(3,6) \cup H_v$} \label{fig torus link inside torus}
\end{figure}

We consider the link $L = T(e,ek) \cup H_v = L_1 \cup \ldots \cup L_e \cup L_{e+1}$. An example is drawn in Figure \ref{fig torus link inside torus} for $e=3, k=2$. We compute the $L^2$-Alexander torsions of its exterior.
Let $\lambda$ denote a meridian of $H_v$ and $\mu$ a preferred longitude of $H_v$. Remark that $\lambda, \mu$ are respectively a longitude and a meridian of the torus on which $T(e,ek)$ is drawn. Let $b_i$ denote the meridians of the components of $T(e,ek)$, as in Figure \ref{fig torus link inside torus}.

\begin{proposition} \label{prop_e,ek_inside_torus}
 The $L^2$-Alexander torsion for the exterior of the $(e+1)$-compo-\\
 nent link 
 $L = T(e,ek) \cup H_v$  is non-zero for all $t>0$ and for all admissible triples of the form
 $(G_L, (n_1,\ldots,n_e,n_{e+1}) \circ \alpha_L, \gamma)$ such that
$\gamma(\lambda \mu^{k} )$ has infinite order in $G$. One has:
$$ T^{(2)}_{L, (n_1,\ldots,n_e,n_{e+1})}(\gamma)(t)
\ \dot{=} \  \max(1,t)^{(e-1)|n_{e+1} + k(n_1 + \ldots + n_e)|}.$$
\end{proposition}

\begin{proof}
Let $KC = T(e,0) \cup H'_v = K_1 \cup \ldots \cup K_e \cup K_{e+1}$ be the $(e+1)$-component keychain link (see Figure \ref{fig keychain 4}).
Then the exteriors $M_{L}$ and $M_{KC}$ are homeomorphic, by a sequence of $k$  twists of the solid torus $S^3 \setminus V(K_{e+1}) \cong S^3 \setminus V(L_{e+1})$.
The induced group isomorphism $\psi: G_{KC} \to G_L$ relates the generators written in the two figures in the following way:
\begin{align*} 
\Z^{e+1}  \overset{\alpha_{KC}}{\longleftarrow} G_{KC}  &\overset{\psi}{\longrightarrow} G_L \overset{\alpha_L}{\longrightarrow} \Z^{e+1}\\
(1,\ldots,0,0)   \longleftarrow a_1  &\longleftrightarrow b_1 \longmapsto (1,\ldots,0,0) \\
&   \ \ \vdots  \\
(0,\ldots,1,0) \longleftarrow   a_{e}  &\longleftrightarrow b_{e} \longmapsto (0,\ldots,1,0) \\
(0,\ldots,0,1)  \longleftarrow a_{e+1}  &\longleftrightarrow \lambda \mu^{k}  \longmapsto (k,\ldots,k,1)
\end{align*} 
Thus, for all integers $n_1,\ldots,n_e,n_{e+1}$, 
$$(n_1,\ldots,n_e,n_{e+1}) \circ \alpha_L \circ \psi = 
(n_1,\ldots,n_e,n_{e+1} + k n_1 + \ldots + k n_e) \circ \alpha_{KC}.$$
Let $\phi$ denote $(n_1,\ldots,n_e,n_{e+1}) \circ \alpha_L$.
Since $(G_{KC}, \phi \circ \psi, \gamma \circ \psi)$ is an admissible triple and since $\gamma( \psi(a_{e+1})) = \gamma(\lambda \mu^{k} )$ has infinite order in $G$, it follows  from Theorem \ref{torsion_keychain} that
$C^{(2)}_*(M_{KC}, \phi \circ \psi, \gamma \circ \psi,t)$ is weakly acyclic and of determinant class and 
 \begin{align*}
 T^{(2)}(M_{KC}, \phi \circ \psi, \gamma \circ \psi)(t)  
& \  \dot{=}  \ \max(1,t)^{(e-1) | \phi(\psi(a_{e+1}))|} \\
& = \max(1,t)^{(e-1) |n_{e+1} + k n_1 + \ldots + k n_e|}.
 \end{align*}
Since $M_L$ and $M_{KC}$ are homeomorphic, they are simple homotopy equivalent and the result follows from Proposition \ref{prop L2 torsion simple homotopy}.
\end{proof}

\subsection{The link $T(p,q) \cup H_v \cup H_h$}
\label{sec:tpqhh}

We consider the $3$-component link $L = T(p,q) \cup H_v \cup H_h$ where $p \neq 0$ and $p,q$ are relatively prime. An example for $p=3, q=4$ is drawn in Figure \ref{fig torus knot inside thick}.

\begin{figure}[!h]
\centering
\begin{tikzpicture}[every path/.style={string ,black} , every node/.style={transform shape , knot crossing , inner sep=1.5 pt } ]
\begin{scope}[scale=0.8]

\begin{scope}[xshift=0cm, yshift =0cm, rotate=0,scale=1]

\begin{scope}[xshift=0cm,rotate=0,scale=1]

\coordinate (ch) at (-2,-1) ;
\coordinate (cb) at (-2,-2) ;
\coordinate (t1) at (-3,2) ;
\coordinate (t2) at (3,2) ;
\coordinate (t3) at (-3,1) ;
\coordinate (t4) at (3,1) ;
\coordinate (t5) at (-3,-1) ;
\coordinate (t6) at (3,-1) ;
\coordinate (t7) at (-3,-2) ;
\coordinate (t8) at (3,-2) ;
\coordinate (k1) at (-3,1.8) ;
\coordinate (k2) at (3,1.8) ;
\coordinate (k3) at (-3,1.2) ;
\coordinate (k4) at (3,1.2) ;
\coordinate (k5) at (-3,-1.2) ;
\coordinate (k6) at (3,-1.2) ;
\coordinate (k7) at (-3,-1.8) ;
\coordinate (k8) at (3,-1.8) ;
\coordinate (k9) at (-3,1.5) ;
\coordinate (k10) at (3,1.5) ;
\coordinate (k11) at (3,-1.5) ;
\coordinate (k12) at (-3,-1.5) ;
\draw (k2) ..controls +(2.3,0) and +(2.3,0).. (k8) -- (-0.2,-1.8);
\draw[->] (-0.4,-1.8) -- (k7) ..controls +(-2.3,0) and +(-2.3,0).. (k1) ;
\draw (k4) ..controls +(1.3,0) and +(1.3,0).. (k6) --(-0.2,-1.2);
\draw[->] (-0.4,-1.2) -- (k5) ..controls +(-1.3,0) and +(-1.3,0).. (k3) ;
\draw  (k10) ..controls +(1.8,0) and +(1.8,0).. (k11) -- (-0.2,-1.5);
\draw[->] (-0.4,-1.5) -- (k12) ..controls +(-1.8,0) and +(-1.8,0).. (k9) ;

\begin{scope}[xshift=-3cm]
\coordinate (kh) at (0,1.8) ;
\coordinate (km) at (0,1.5) ;
\coordinate (kb) at (0,1.2) ;
\coordinate (khi) at (0.5,2) ;
\coordinate (kbi) at (1,1) ;
\coordinate (kh') at (1.5,1.8) ;
\coordinate (km') at (1.5,1.5) ;
\coordinate (kb') at (1.5,1.2) ;
\draw  (kh) ..controls +(0.1,0) and +(-0.1,0).. (kb') ;
\draw [style=dashed] (km) ..controls +(0.1,0) and +(-0.1,0).. (khi) ;
\draw  (khi) ..controls +(0.1,0) and +(-0.1,0).. (kh') ;
\draw  (kb) ..controls +(0.1,0) and +(-0.1,0).. (kbi) ;
\draw [style=dashed] (kbi) ..controls +(0.1,0) and +(-0.1,0).. (km') ;
\end{scope}

\begin{scope}[xshift=-1.5cm]
\coordinate (kh) at (0,1.8) ;
\coordinate (km) at (0,1.5) ;
\coordinate (kb) at (0,1.2) ;
\coordinate (khi) at (0.5,2) ;
\coordinate (kbi) at (1,1) ;
\coordinate (kh') at (1.5,1.8) ;
\coordinate (km') at (1.5,1.5) ;
\coordinate (kb') at (1.5,1.2) ;
\draw  (kh) ..controls +(0.1,0) and +(-0.1,0).. (kb') ;
\draw [style=dashed] (km) ..controls +(0.1,0) and +(-0.1,0).. (khi) ;
\draw  (khi) ..controls +(0.1,0) and +(-0.1,0).. (kh') ;
\draw  (kb) ..controls +(0.1,0) and +(-0.1,0).. (kbi) ;
\draw [style=dashed] (kbi) ..controls +(0.1,0) and +(-0.1,0).. (km') ;
\end{scope}

\begin{scope}[xshift=0cm]
\coordinate (kh) at (0,1.8) ;
\coordinate (km) at (0,1.5) ;
\coordinate (kb) at (0,1.2) ;
\coordinate (khi) at (0.5,2) ;
\coordinate (kbi) at (1,1) ;
\coordinate (kh') at (1.5,1.8) ;
\coordinate (km') at (1.5,1.5) ;
\coordinate (kb') at (1.5,1.2) ;
\draw  (kh) ..controls +(0.1,0) and +(-0.1,0).. (kb') ;
\draw [style=dashed] (km) ..controls +(0.1,0) and +(-0.1,0).. (khi) ;
\draw  (khi) ..controls +(0.1,0) and +(-0.1,0).. (kh') ;
\draw  (kb) ..controls +(0.1,0) and +(-0.1,0).. (kbi) ;
\draw [style=dashed] (kbi) ..controls +(0.1,0) and +(-0.1,0).. (km') ;
\end{scope}

\begin{scope}[xshift=1.5cm]
\coordinate (kh) at (0,1.8) ;
\coordinate (km) at (0,1.5) ;
\coordinate (kb) at (0,1.2) ;
\coordinate (khi) at (0.5,2) ;
\coordinate (kbi) at (1,1) ;
\coordinate (kh') at (1.5,1.8) ;
\coordinate (km') at (1.5,1.5) ;
\coordinate (kb') at (1.5,1.2) ;
\draw  (kh) ..controls +(0.1,0) and +(-0.1,0).. (kb') ;
\draw [style=dashed] (km) ..controls +(0.1,0) and +(-0.1,0).. (khi) ;
\draw  (khi) ..controls +(0.1,0) and +(-0.1,0).. (kh') ;
\draw  (kb) ..controls +(0.1,0) and +(-0.1,0).. (kbi) ;
\draw [style=dashed] (kbi) ..controls +(0.1,0) and +(-0.1,0).. (km') ;
\end{scope}

\coordinate (x9) at (-3,1.4) ;
\coordinate (x10) at (3,1.4) ;
\coordinate (x11) at (3,-1.4) ;
\coordinate (x12) at (-3,-1.4) ;
\coordinate (x13) at (0,-1.4) ;

\draw [->] [color=blue] (x10) ..controls +(1.7,0) and +(1.7,0).. (x11) -- (x13);
\draw [color=blue] (x13) -- (x12) ..controls +(-1.7,0) and +(-1.7,0).. (x9)--(x10) ;

\draw [color=blue] (1,-1.65) node {$H_h$} ;

\begin{scope}[xshift=-2cm,rotate=0,scale=1]
\coordinate (yh) at (0,-0.7) ;
\coordinate (yb) at (0,-2.3) ;
\draw [<-] [very thick, color=brown] (0.3,-1) -- (0.3,-0.7) -- (0,-0.7);
\draw [very thick, color=brown] (0.3,-0.9) -- (0.3,-1.15);
\draw [very thick, color=brown] (0,-0.7) -- (-0.3,-0.7) -- (-0.3,-2.3) -- (0.3,-2.3) -- (0.3,-2) -- (0.3,-1.85) ;
\draw [very thick, color=brown] (0.5,-0.75) node {$y$} ;
\end{scope}

\draw  (5,1.7) node {$T(p,q)$} ;

\end{scope}

\begin{scope}[xshift=0cm,rotate=0,scale=1]
\coordinate (yh) at (0,-0.7) ;
\coordinate (yb) at (0,-2.3) ;
\draw [color=violet](0.3,-1.1) -- (0.3,-0.5);
\draw [color=violet](0.3,-0.3) -- (0.3,0) -- (-0.3,0) -- (-0.3,-1.5);
\draw [color=violet][<-] (-0.3,-1.5) -- (-0.3,-3) -- (0.3,-3) -- (0.3,-1.9) ;

\draw [color=violet] (-0.65,-2.7) node {$H_v$} ;
\end{scope}

\coordinate (th) at (0,3) ;
\coordinate (td) at (5.5,0) ;
\coordinate (tb) at (0,-3) ;
\coordinate (tg) at (-5.5,0) ;

\draw [color=magenta, very thick] (2,0.4) -- (2,-0.4) -- (-0.2,-0.4);
\draw [->] [color=magenta, very thick] (-0.4,-0.4) -- (-1.5,-0.4) ;
\draw [color=magenta, very thick] (-1.5,-0.4) -- (-2,-0.4) -- (-2,0.4) -- (2,0.4) ;
\draw [color=magenta](-1.5,-0.15) node {$\lambda$} ;
\end{scope}
\end{scope}
\end{tikzpicture}
\caption{The link $T(3,4) \cup H_v \cup H_h$} \label{fig torus knot inside thick}
\end{figure}
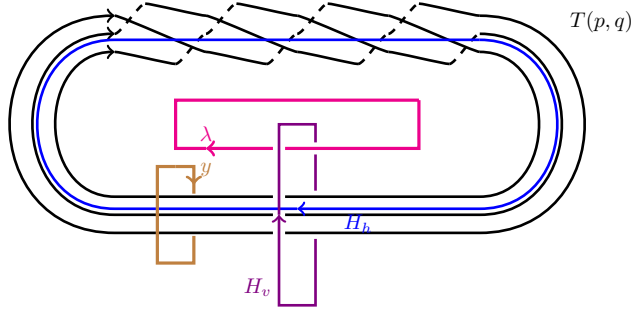

 Tubular neighbourhoods of $H_h$ and $H_v$ have a torus $T$ as a common boundary. 
The manifold $M_{H} = S^3 \setminus (V(H_v) \cup V(H_h)) $ is homeomorphic to a thickened torus $T \times [-1;1]$. We identify $T \cong T \times \{0\}$ to the torus on which the torus knot $T(p,q)$ is drawn.
The space $Z = T \setminus V(T(p,q))$ is homeomorphic to an annulus. Let $\delta$ be a curve that generates $\pi_1(Z)$. The curve $\delta$ is thus locally parallel to the torus knot $T(p,q)$. See Figure \ref{fig torus delta} for clarity.

\begin{figure}[!h]
\centering
\begin{tikzpicture}[every path/.style={string } , every node/.style={transform shape , knot crossing , inner sep=1.5 pt } ]

\begin{scope}[scale=0.8]

\begin{scope}[xshift=0cm,rotate=0,scale=1]

\draw (-2,-0.8) node {$T$} ;

\coordinate (ch) at (-2,-1) ;
\coordinate (cb) at (-2,-2) ;
\draw [style=dotted] (ch) ..controls +(0.5,0) and +(0.5,0).. (cb) ;
\draw (ch) ..controls +(-0.5,0) and +(-0.5,0).. (cb) ;

\coordinate (t1) at (-3,2) ;
\coordinate (t2) at (3,2) ;
\coordinate (t3) at (-3,1) ;
\coordinate (t4) at (3,1) ;
\coordinate (t5) at (-3,-1) ;
\coordinate (t6) at (3,-1) ;
\coordinate (t7) at (-3,-2) ;
\coordinate (t8) at (3,-2) ;
\draw (t1) -- (t2) ..controls +(2.5,0) and +(2.5,0).. (t8) -- (t7) ..controls +(-2.5,0) and +(-2.5,0).. (t1) ;
\draw (t3) -- (t4) ..controls +(1,0) and +(1,0).. (t6) -- (t5) ..controls +(-1,0) and +(-1,0).. (t3) ;

\coordinate (k1) at (-3,1.8) ;
\coordinate (k2) at (3,1.8) ;
\coordinate (k3) at (-3,1.2) ;
\coordinate (k4) at (3,1.2) ;
\coordinate (k5) at (-3,-1.2) ;
\coordinate (k6) at (3,-1.2) ;
\coordinate (k7) at (-3,-1.8) ;
\coordinate (k8) at (3,-1.8) ;
\coordinate (k9) at (-3,1.5) ;
\coordinate (k10) at (3,1.5) ;
\coordinate (k11) at (3,-1.5) ;
\coordinate (k12) at (-3,-1.5) ;
\draw (k2) ..controls +(2.3,0) and +(2.3,0).. (k8) -- (k7) ..controls +(-2.3,0) and +(-2.3,0).. (k1) ;
\draw (k4) ..controls +(1.3,0) and +(1.3,0).. (k6) -- (k5) ..controls +(-1.3,0) and +(-1.3,0).. (k3) ;
\draw [style=dotted] (k10) ..controls +(1.8,0) and +(1.8,0).. (k11) -- (k12) ..controls +(-1.8,0) and +(-1.8,0).. (k9) ;

\begin{scope}[xshift=-3cm]
\coordinate (kh) at (0,1.8) ;
\coordinate (km) at (0,1.5) ;
\coordinate (kb) at (0,1.2) ;
\coordinate (khi) at (0.5,2) ;
\coordinate (kbi) at (1,1) ;
\coordinate (kh') at (1.5,1.8) ;
\coordinate (km') at (1.5,1.5) ;
\coordinate (kb') at (1.5,1.2) ;
\draw  (kh) ..controls +(0.1,0) and +(-0.1,0).. (kb') ;
\draw [style=dotted] (km) ..controls +(0.1,0) and +(-0.1,0).. (khi) ;
\draw  (khi) ..controls +(0.1,0) and +(-0.1,0).. (kh') ;
\draw  (kb) ..controls +(0.1,0) and +(-0.1,0).. (kbi) ;
\draw [style=dotted] (kbi) ..controls +(0.1,0) and +(-0.1,0).. (km') ;
\coordinate (kbr) at (0.25,1) ;
\coordinate (khr) at (1.25,2) ;
\draw [style=dotted, color=red] (kbr) ..controls +(0.1,0) and +(-0.1,0).. (khr) ;
\draw [color=red] (0,1.1) ..controls +(0.1,0) and +(-0.1,0).. (kbr) ;
\draw [->] [color=red] (khr) ..controls +(0.1,0) and +(-0.1,0).. (1.5,1.9) ;
\draw [->] [color=red] (0,1.9) ..controls +(0.3,0) and +(-0.1,0).. (1.5,1.4) ;
\draw [->] [color=red] (0,1.4) ..controls +(0.1,0) and +(-0.1,0).. (1.5,1.1) ;
\end{scope}

\begin{scope}[xshift=-1.5cm]
\coordinate (kh) at (0,1.8) ;
\coordinate (km) at (0,1.5) ;
\coordinate (kb) at (0,1.2) ;
\coordinate (khi) at (0.5,2) ;
\coordinate (kbi) at (1,1) ;
\coordinate (kh') at (1.5,1.8) ;
\coordinate (km') at (1.5,1.5) ;
\coordinate (kb') at (1.5,1.2) ;
\draw  (kh) ..controls +(0.1,0) and +(-0.1,0).. (kb') ;
\draw [style=dotted] (km) ..controls +(0.1,0) and +(-0.1,0).. (khi) ;
\draw  (khi) ..controls +(0.1,0) and +(-0.1,0).. (kh') ;
\draw  (kb) ..controls +(0.1,0) and +(-0.1,0).. (kbi) ;
\draw [style=dotted] (kbi) ..controls +(0.1,0) and +(-0.1,0).. (km') ;
\coordinate (kbr) at (0.25,1) ;
\coordinate (khr) at (1.25,2) ;
\draw [style=dotted, color=red] (kbr) ..controls +(0.1,0) and +(-0.1,0).. (khr) ;
\draw [color=red] (0,1.1) ..controls +(0.1,0) and +(-0.1,0).. (kbr) ;
\draw [->] [color=red] (khr) ..controls +(0.1,0) and +(-0.1,0).. (1.5,1.9) ;
\draw [->] [color=red] (0,1.9) ..controls +(0.3,0) and +(-0.1,0).. (1.5,1.4) ;
\draw [->] [color=red] (0,1.4) ..controls +(0.1,0) and +(-0.1,0).. (1.5,1.1) ;
\end{scope}

\begin{scope}[xshift=0cm]
\coordinate (kh) at (0,1.8) ;
\coordinate (km) at (0,1.5) ;
\coordinate (kb) at (0,1.2) ;
\coordinate (khi) at (0.5,2) ;
\coordinate (kbi) at (1,1) ;
\coordinate (kh') at (1.5,1.8) ;
\coordinate (km') at (1.5,1.5) ;
\coordinate (kb') at (1.5,1.2) ;
\draw  (kh) ..controls +(0.1,0) and +(-0.1,0).. (kb') ;
\draw [style=dotted] (km) ..controls +(0.1,0) and +(-0.1,0).. (khi) ;
\draw  (khi) ..controls +(0.1,0) and +(-0.1,0).. (kh') ;
\draw  (kb) ..controls +(0.1,0) and +(-0.1,0).. (kbi) ;
\draw [style=dotted] (kbi) ..controls +(0.1,0) and +(-0.1,0).. (km') ;
\coordinate (kbr) at (0.25,1) ;
\coordinate (khr) at (1.25,2) ;
\draw [style=dotted, color=red] (kbr) ..controls +(0.1,0) and +(-0.1,0).. (khr) ;
\draw [color=red] (0,1.1) ..controls +(0.1,0) and +(-0.1,0).. (kbr) ;
\draw [->] [color=red] (khr) ..controls +(0.1,0) and +(-0.1,0).. (1.5,1.9) ;
\draw [->] [color=red] (0,1.9) ..controls +(0.3,0) and +(-0.1,0).. (1.5,1.4) ;
\draw [->] [color=red] (0,1.4) ..controls +(0.1,0) and +(-0.1,0).. (1.5,1.1) ;
\end{scope}

\begin{scope}[xshift=1.5cm]
\coordinate (kh) at (0,1.8) ;
\coordinate (km) at (0,1.5) ;
\coordinate (kb) at (0,1.2) ;
\coordinate (khi) at (0.5,2) ;
\coordinate (kbi) at (1,1) ;
\coordinate (kh') at (1.5,1.8) ;
\coordinate (km') at (1.5,1.5) ;
\coordinate (kb') at (1.5,1.2) ;
\draw  (kh) ..controls +(0.1,0) and +(-0.1,0).. (kb') ;
\draw [style=dotted] (km) ..controls +(0.1,0) and +(-0.1,0).. (khi) ;
\draw  (khi) ..controls +(0.1,0) and +(-0.1,0).. (kh') ;
\draw  (kb) ..controls +(0.1,0) and +(-0.1,0).. (kbi) ;
\draw [style=dotted] (kbi) ..controls +(0.1,0) and +(-0.1,0).. (km') ;
\coordinate (kbr) at (0.25,1) ;
\coordinate (khr) at (1.25,2) ;
\draw [style=dotted, color=red] (kbr) ..controls +(0.1,0) and +(-0.1,0).. (khr) ;
\draw [color=red] (0,1.1) ..controls +(0.1,0) and +(-0.1,0).. (kbr) ;
\draw [->] [color=red] (khr) ..controls +(0.1,0) and +(-0.1,0).. (1.5,1.9) ;
\draw [->] [color=red] (0,1.9) ..controls +(0.3,0) and +(-0.1,0).. (1.5,1.4) ;
\draw [->] [color=red] (0,1.4) ..controls +(0.1,0) and +(-0.1,0).. (1.5,1.1) ;
\end{scope}

\coordinate (k1') at (-3,1.9) ;
\coordinate (k2') at (3,1.9) ;
\coordinate (k3') at (-3,1.1) ;
\coordinate (k4') at (3,1.1) ;
\coordinate (k5') at (-3,-1.1) ;
\coordinate (k6') at (3,-1.1) ;
\coordinate (k7') at (-3,-1.9) ;
\coordinate (k8') at (3,-1.9) ;
\coordinate (k9') at (-3,1.4) ;
\coordinate (k10') at (3,1.4) ;
\coordinate (k11') at (3,-1.4) ;
\coordinate (k12') at (-3,-1.4) ;
\draw [->] [color=red] (k2') ..controls +(2.4,0) and +(2.4,0).. (k8') -- (k7') ..controls +(-2.4,0) and +(-2.4,0).. (k1') ;
\draw [->] [color=red] (k4') ..controls +(1.2,0) and +(1.2,0).. (k6') -- (k5') ..controls +(-1.2,0) and +(-1.2,0).. (k3') ;
\draw [->] [color=red] (k10') ..controls +(1.7,0) and +(1.7,0).. (k11') -- (k12') ..controls +(-1.7,0) and +(-1.7,0).. (k9') ;
\draw [color=red] (0,-1.31) node {$\delta$} ;
\end{scope}
\end{scope}
\end{tikzpicture}
\caption{The generator $\delta$ of $\pi_1(Z)$} \label{fig torus delta}
\end{figure}
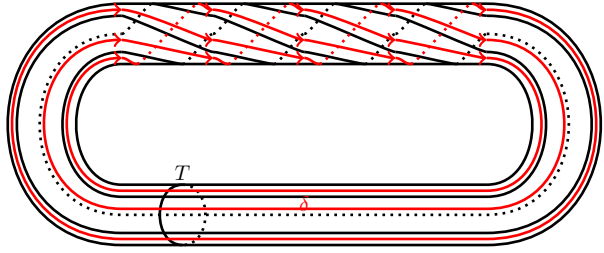

\begin{proposition} \label{prop_p,q_inside_thick}
 The $L^2$-Alexander torsion for the exterior of the $3$-component link 
 $L = T(p,q) \cup H_v \cup H_h$  is non-zero for all $t>0$ and for all admissible triples 
 $(G_L, (n_1,n_2,n_3) \circ \alpha_L, \gamma)$ such that the homotopy class of the curve $\delta$ is sent by $\gamma$ to an element of infinite order. One has:
$$ T^{(2)}_{L, (n_1,n_2,n_3)}(\gamma)(t)
\ \dot{=} \  \max(1,t)^{|pq n_1 + p n_2 + q n_3|}.$$
\end{proposition}

\begin{proof}
The torus $T$ separates $M_{H}$ in two thickened tori $N_1 = V(H_v)\setminus H_v$ and \\
$N_2 = V(H_h) \setminus H_h$.
Let $X= M_{L}$, $A=  N_1 \cup T \setminus V(T(p,q))$, $B = N_2 \cup T \setminus V(T(p,q))$ and $Z = T \setminus V(T(p,q))$, so that $X = A \cup B$ and $ Z = A \cap B$, and $X,A,B,Z$ are path connected. We pick a base point $pt \in Z$ for all the following fundamental groups.

The space $Z$ is an annulus, and its group $\pi_Z = \pi_1(Z)$ is isomorphic to $\Z$ and is generated by an element $\delta$ that runs between the $p$ strands of $T(p,q)$.
The space $A$ is homeomorphic to a thickened torus, by filling the missing surface lines of $V(T(p,q))$. Let $(y,\lambda)$ be a preferred meridian-longitude system of $A$, as in Figure \ref{fig torus knot inside thick}. Note that $\lambda$ acts as a meridian of the unknot $H_v$. The group $\pi_A = \pi_1(A)$ has the presentation $\langle y, \lambda | y \lambda = \lambda y \rangle$ and is isomorphic to $\Z^2$.
Similarly, the space $B$ is homeomorphic to a thickened torus, by filling the missing surface lines of $V(T(p,q))$. Let $(m,c)$ be a preferred meridian-longitude system of $A$. Note that $m$ acts as a meridian of the unknot $H_h$. The group $\pi_B = \pi_1(B)$ has the presentation $\langle m,c | m c =c m \rangle$ and is isomorphic to $\Z^2$.
The element $z$ is sent to $\lambda^p y^q $ in $\pi_A$ and to $c^p m^q$ in $\pi_B$.
Thus, by the Seifert van Kampen theorem, the group $G_L = \pi_1(X)$ admits the presentation
$$ \langle y, \lambda, c, m | \lambda^p y^q = c^p m^q, y \lambda = \lambda y,  m c =c m  \rangle.
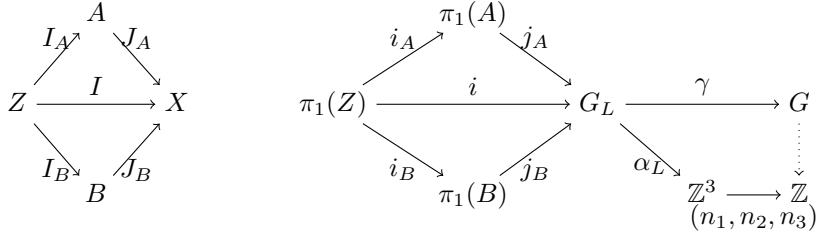
$$ 

\begin{tikzpicture}
[description/.style={fill=white,inner sep=2pt}] 
\matrix(a)[matrix of math nodes, row sep=2em, column sep=1.5em, text height=1.5ex, text depth=0.25ex] 
{ & A\\ Z & & X\\ & B\\}; 
\path[->](a-2-1) edge node[below]{$I_B$} (a-3-2); 
\path[->](a-2-1) edge node[above]{$I_A$} (a-1-2);  
\path[->](a-1-2) edge node[above]{$J_A$} (a-2-3); 
\path[->](a-3-2) edge node[below]{$J_B$} (a-2-3); 
\path[->](a-2-1) edge node[above]{$I$} (a-2-3);  

\begin{scope}[xshift=6cm,rotate=0,scale=1]
[description/.style={fill=white,inner sep=2pt}] 
\matrix(a)[matrix of math nodes, row sep=2em, column sep=2em, text height=1.5ex, text depth=0.25ex] 
{ & \pi_1(A)\\ \pi_1(Z) & & G_{L} & & G\\ & \pi_1(B) & & \Z^{3} & \Z\\}; 
\path[->](a-2-1) edge node[below]{$i_B$} (a-3-2); 
\path[->](a-2-1) edge node[above]{$i_A$} (a-1-2);  
\path[->](a-1-2) edge node[above]{$j_A$} (a-2-3); 
\path[->](a-3-2) edge node[below]{$j_B$} (a-2-3); 
\path[->](a-2-1) edge node[above]{$i$} (a-2-3);
\path[->](a-2-3) edge node[above]{$\gamma$} (a-2-5);
\path[->](a-2-3) edge node[below]{$\alpha_{L}$} (a-3-4);
\path[->](a-3-4) edge node[below]{$(n_1,n_2,n_3)$} (a-3-5);
\path[->, dotted](a-2-5) edge  (a-3-5);
\end{scope}
\end{tikzpicture}

Let $\phi$ denote the homomorphism  $(n_1,n_2,n_3) \circ  \alpha_L$.
We assume that the homotopy class of $\delta$ is sent by $\gamma$ to an element of infinite order, i.e. $\gamma \circ i(\delta) = \gamma(c^p m^q) = \gamma(\lambda^p y^q)$ has infinite order in $G$. Therefore $\gamma(\pi_A)$ and $\gamma(\pi_B)$ are infinite and it follows from Corollaries \ref{torsion solid torus} and \ref{torsion torus} that
 the three $\NN(G)$-cellular chain complexes 
$$C^{(2)}_*(Z,\phi \circ i, \gamma \circ i,t),
C^{(2)}_*(A,\phi \circ j_A, \gamma \circ j_A,t),
C^{(2)}_*(B,\phi \circ j_B, \gamma \circ j_B,t)$$ are
weakly acyclic and of determinant class, and
$$T^{(2)}(A,\phi \circ j_A, \gamma \circ j_A)(t) \ \dot{=} \ 
T^{(2)}(B,\phi \circ j_B, \gamma \circ j_B)(t) \ \dot{=} \  1,$$
$$ T^{(2)}(Z,\phi \circ i, \gamma \circ i)(t) \ \dot{=} \  \max(1,t)^{-|\phi(c^p m^q)|}.$$

Hence, by Proposition \ref{prop mayer vietoris}, 
  $C^{(2)}_*(M_{L},(n_1,n_2,n_3) \circ  \alpha_L, \gamma,t)$ is weakly acyclic and of determinant class as well, and
$$T^{(2)}_{L,(n_1,n_2,n_3)}(\gamma)(t)
\ \dot{=} \  \max(1,t)^{|\phi(c^p m^q)|}
= \max(1,t)^{|pq n_1 + p n_2 + q n_3|}.
$$
\end{proof}

\subsection{The link $T(ep,eq) \cup H_v \cup H_h$}
\label{sec:tepeqhh}

We can now compute the $L^2$-Alexander torsions for a general link 
$L = T(ep,eq) \cup H_v \cup H_h$ by using the fact that the torus link $T(ep,eq)$ is a $(e,epq)$-cable on the torus knot
$K = T(p,q)$ (one can find a proof of this result due to P. Feller in \cite[Section A.2]{BAthesis}). We refer to \cite{BZ} for the definition of a cable link.

\begin{figure}[!h]
\centering
\begin{tikzpicture}[every path/.style={string ,black} , every node/.style={transform shape , knot crossing , inner sep=3.5 pt } ]

\begin{scope}[scale=0.9]

\begin{scope}[xshift=0cm, yshift=0cm,rotate=0,scale=0.8]

\begin{scope}[xshift=-4cm, yshift=-1.15cm,rotate=0,scale=1]
\draw [color=purple] (0,-0.4)--(0,-0.55);
\draw [color=purple][->] (0,-0.65)--(0,-0.9);
\draw [color=purple] (0.3,-0.9) node {$b_1$};
\end{scope}

\begin{scope}[xshift=-5cm, yshift=-0.79cm,rotate=0,scale=1]
\draw [color=purple] (0,-0.4)--(0,-0.55);
\draw [color=purple][->] (0,-0.65)--(0,-0.9);
\draw [color=purple] (0.3,-0.4) node {$b_2$};
\end{scope}

\draw [scale=1.2,color=red] (1.5,-3) node {$\partial(V(H_K)) \cong T_K \cong \partial(V(K))$} ;

\draw [scale=1.2,color=red] (-5,-4) node {$T_K$} ;

\draw [color=red] (-1.7,-2.2) node {$H_K$} ;
\draw [color=magenta] (-3.5,-0.7) node {$l_K$} ;

\begin{scope}[xshift=-6cm, yshift=0cm,rotate=0,scale=0.7]

\begin{scope}[xshift=0cm, yshift=0cm,rotate=0,scale=0.5]

\begin{scope}[xshift=10cm, yshift=0cm,rotate=0,scale=1]
\draw[->] [color=red](1,-8) -- (-1,-8) -- (-1,-3);
\draw [color=red](-1,-3) -- (-1,-2) -- (1,-2) -- (1,-2.4);
\draw [color=red](1,-2.8) -- (1,-3.6);
\draw [color=red](1,-4.2) -- (1,-4.8);
\draw [color=red](1,-5.3) -- (1,-8);

\draw [color=magenta][->](-1.2,-2.6) -- (-2,-2.6) -- (-2,-1.4) -- (0,-1.4);
\draw [color=magenta](0,-1.4) -- (2,-1.4) -- (2,-2.6) -- (-0.8,-2.6);

\end{scope}

\draw (20,0) ..controls +(2,0) and +(2,0).. (20,-4) --(9.2,-4);
\draw (8.8,-4) -- (0,-4) ..controls +(-2,0) and +(-2,0).. (0,0);
\draw [color=blue](20,1) ..controls +(3,0) and +(3,0).. (20,-5) --(9.2,-5);
\draw  [color=blue](8.8,-5) -- (0,-5) ..controls +(-3,0) and +(-3,0).. (0,1);

\draw (0,0) -- (2,0);
\draw (9,0) -- (10,0);
\draw (17,0) -- (20,0);
\draw (5,1) -- (6,1);
\draw (13,1) -- (14,1);

\draw[color=blue] (0,1) -- (2,1);
\draw[color=blue] (9,1) -- (10,1);
\draw[color=blue] (17,1) -- (20,1);
\draw[color=blue] (5,0) -- (6,0);
\draw[color=blue] (13,0) -- (14,0);

\begin{scope}[xshift=2cm, yshift=-3cm,rotate=0,scale=1]
\draw [color=blue](0,4) -- (1,4) ..controls +(0.5,0) and +(-0.5,0).. (2,3) -- (3,3);
\draw (0,3) -- (1,3);
\node (ca) at (1,3) { };
\node (cc) at (1.5,3.5) { };
\node (cb) at (2,4) { };
\draw (ca.center) .. controls (ca.2 north east) and (cc.2 south west) .. (cc);
\draw (cc) .. controls (cc.2 north east) and (cb.2 south west) .. (cb.center);
\draw (2,4) -- (3,4);
\end{scope}

\begin{scope}[xshift=6cm, yshift=-3cm,rotate=0,scale=1]
\draw (0,4) -- (1,4) ..controls +(0.5,0) and +(-0.5,0).. (2,3) -- (3,3);
\draw [color=blue](0,3) -- (1,3);
\node (ca) at (1,3) { };
\node (cc) at (1.5,3.5) { };
\node (cb) at (2,4) { };
\draw [color=blue](ca.center) .. controls (ca.2 north east) and (cc.2 south west) .. (cc);
\draw [color=blue](cc) .. controls (cc.2 north east) and (cb.2 south west) .. (cb.center);
\draw [color=blue](2,4) -- (3,4);
\end{scope}

\begin{scope}[xshift=10cm, yshift=-3cm,rotate=0,scale=1]
\draw [color=blue](0,4) -- (1,4) ..controls +(0.5,0) and +(-0.5,0).. (2,3) -- (3,3);
\draw (0,3) -- (1,3);
\node (ca) at (1,3) { };
\node (cc) at (1.5,3.5) { };
\node (cb) at (2,4) { };
\draw (ca.center) .. controls (ca.2 north east) and (cc.2 south west) .. (cc);
\draw (cc) .. controls (cc.2 north east) and (cb.2 south west) .. (cb.center);
\draw (2,4) -- (3,4);
\end{scope}

\begin{scope}[xshift=14cm, yshift=-3cm,rotate=0,scale=1]
\draw (0,4) -- (1,4) ..controls +(0.5,0) and +(-0.5,0).. (2,3) -- (3,3);
\draw [color=blue](0,3) -- (1,3);
\node (ca) at (1,3) { };
\node (cc) at (1.5,3.5) { };
\node (cb) at (2,4) { };
\draw [color=blue](ca.center) .. controls (ca.2 north east) and (cc.2 south west) .. (cc);
\draw [color=blue](cc) .. controls (cc.2 north east) and (cb.2 south west) .. (cb.center);
\draw [color=blue](2,4) -- (3,4);
\end{scope}

\end{scope}

\end{scope}

\draw [color=red] (4,0.7) node {$K = T(2,1)$} ;
\draw [color=blue] (6,-1.2) node {$H_h$} ;

\draw [very thick, color=violet] (3.4,-2.1) node {$H_v$} ;

\begin{scope}[xshift=5cm, yshift =-0.7cm, rotate=0,scale=0.6]

\begin{scope}[xshift=0cm,rotate=0,scale=1]

\coordinate (ch) at (-2,-1) ;
\coordinate (cb) at (-2,-2) ;

\coordinate (t1) at (-3,2) ;
\coordinate (t2) at (3,2) ;
\coordinate (t3) at (-3,1) ;
\coordinate (t4) at (3,1) ;
\coordinate (t5) at (-3,-1) ;
\coordinate (t6) at (3,-1) ;
\coordinate (t7) at (-3,-2) ;
\coordinate (t8) at (3,-2) ;

\coordinate (k1) at (-3,1.8) ;
\coordinate (k2) at (3,1.8) ;
\coordinate (k3) at (-3,1.2) ;
\coordinate (k4) at (3,1.2) ;
\coordinate (k5) at (-3,-1.2) ;
\coordinate (k6) at (3,-1.2) ;
\coordinate (k7) at (-3,-1.8) ;
\coordinate (k8) at (3,-1.8) ;
\coordinate (k9) at (-3,1.5) ;
\coordinate (k10) at (3,1.5) ;
\coordinate (k11) at (3,-1.5) ;
\coordinate (k12) at (-3,-1.5) ;
\draw [color=red](k2) ..controls +(2.3,0) and +(2.3,0).. (k8) -- (-0.4,-1.8);
\draw[->][color=red] (-0.4,-1.8) -- (k7) ..controls +(-2.3,0) and +(-2.3,0).. (k1) ;
\draw (k4)[color=red] ..controls +(1.3,0) and +(1.3,0).. (k6) --(-0.4,-1.2);
\draw[->][color=red] (-0.4,-1.2) -- (k5) ..controls +(-1.3,0) and +(-1.3,0).. (k3) ;

\coordinate (x9) at (-3,1.4) ;
\coordinate (x10) at (3,1.4) ;
\coordinate (x11) at (3,-1.4) ;
\coordinate (x12) at (-3,-1.4) ;
\coordinate (x13) at (0,-1.4) ;

\draw [->] [color=blue] (x10) ..controls +(1.7,0) and +(1.7,0).. (x11) -- (x13);
\draw [color=blue] (x13) -- (x12) ..controls +(-1.7,0) and +(-1.7,0).. (x9)--(x10) ;

\begin{scope}[xshift=-2cm,rotate=0,scale=1]
\coordinate (yh) at (0,-0.7) ;
\coordinate (yb) at (0,-2.3) ;
\draw [<-] [very thick, color=violet] (0.3,-1) -- (0.3,-0.6) -- (0,-0.6);
\draw [very thick, color=violet] (0.3,-0.9) -- (0.3,-1.15);
\draw [very thick, color=violet] (0,-0.6) -- (-0.3,-0.6) -- (-0.3,-2.3) -- (0.3,-2.3) -- (0.3,-2) -- (0.3,-1.85) ;
\end{scope}

\end{scope}

\draw[color=red] (-3,1.8) -- (3,1.2);
\draw[color=red] (-3,1.2) -- (-1,1.4); 
\draw[color=red] (1,1.6) -- (3,1.8);

\coordinate (th) at (0,3) ;
\coordinate (td) at (5.5,0) ;
\coordinate (tb) at (0,-3) ;
\coordinate (tg) at (-5.5,0) ;

\end{scope}

\begin{scope}[xshift=-4cm, yshift=-6cm,rotate=0,scale=1]

\begin{scope}[xshift=0cm, yshift=0cm,rotate=0,scale=0.5]

\draw (20,0) ..controls +(2,0) and +(2,0).. (20,-4) --(9.2,-4);
\draw (9.2,-4) -- (0,-4) ..controls +(-2,0) and +(-2,0).. (0,0);
\draw [color=blue](20,1) ..controls +(3,0) and +(3,0).. (20,-5) --(9.2,-5);
\draw  [color=blue](9.2,-5) -- (0,-5) ..controls +(-3,0) and +(-3,0).. (0,1);
\draw (20,2) ..controls +(4,0) and +(4,0).. (20,-6) --(9.2,-6);
\draw (9.2,-6) -- (0,-6) ..controls +(-4,0) and +(-4,0).. (0,2);
\draw [color=blue](20,3) ..controls +(5,0) and +(5,0).. (20,-7) --(9.2,-7);
\draw  [color=blue](9.2,-7) -- (0,-7) ..controls +(-5,0) and +(-5,0).. (0,3);

\draw[color=red,style=dotted] (20,0-0.2) ..controls +(1.8,0) and +(1.8,0).. (20,-3.8) --(9.2,-3.8);
\draw [color=red,style=dotted](9.2,-3.8) -- (0,-3.8) ..controls +(-1.8,0) and +(-1.8,0).. (0,0-0.2);
\draw[color=red,style=dotted] (20,1.2) ..controls +(3.2,0) and +(3.2,0).. (20,-5.2) --(9.2,-5.2);
\draw [color=red,style=dotted](9.2,-5.2) -- (0,-5.2) ..controls +(-3.2,0) and +(-3.2,0).. (0,1.2);
\draw[color=red,style=dotted] (20,0-0.2+2) ..controls +(1.8+2,0) and +(1.8+2,0).. (20,-3.8-2) -- (9.2,-5.8);
\draw [color=red,style=dotted](9.2,-5.8) --(0,-3.8-2) ..controls +(-1.8-2,0) and +(-1.8-2,0).. (0,0-0.2+2);
\draw[color=red,style=dotted] (20,2+1.2) ..controls +(2+3.2,0) and +(2+3.2,0).. (20,-5.2-2) -- (9.2,-7.2);
\draw [color=red,style=dotted](9.2,-7.2) --(0,-5.2-2) ..controls +(-3.2-2,0) and +(-3.2-2,0).. (0,1.2+2);

\draw (0,0) -- (7,0);
\draw [color=blue](0,1) -- (4,1);
\draw (0,2) -- (1,2);
\draw [color=blue](0,3) -- (1,3);
\draw [color=blue](10,0) -- (17,0);
\draw (10,1) -- (14,1);
\draw [color=blue](7,2) -- (11,2);
\draw (4,3) -- (11,3);
\draw (17,2) -- (20,2);
\draw [color=blue](14,3) -- (20,3);

\draw[color=red,style=dotted] (0,-0.2) -- (7,-0.2);
\draw[color=red,style=dotted] (0,1.2) -- (4,1.2);
\draw[color=red,style=dotted] (0,1.8) -- (1,1.8);
\draw[color=red,style=dotted] (0,3.2) -- (1,3.2);
\draw[color=red,style=dotted] (10,0.8) -- (14,0.8);
\draw[color=red,style=dotted] (7,2.2) -- (11,2.2);
\draw[color=red,style=dotted] (17,1.8) -- (20,1.8);
\draw[color=red,style=dotted] (14,3.2) -- (20,3.2);

\begin{scope}[xshift=4cm, yshift=-2.2cm,rotate=0,scale=1, style=dotted]
\draw[color=red] (-3,4) -- (0,4) -- (1,4) ..controls +(0.5,0) and +(-0.5,0).. (2,3) -- (3,3);
\end{scope}

\begin{scope}[xshift=7cm, yshift=-3.2cm,rotate=0,scale=1, style=dotted]
\draw[color=red] (0,4) -- (1,4) ..controls +(0.5,0) and +(-0.5,0).. (2,3) -- (3,3) -- (13,3);
\end{scope}

\begin{scope}[xshift=11cm, yshift=-0.8cm,rotate=0,scale=1, style=dotted]
\draw[color=red] (-10,4) -- (0,4) -- (1,4) ..controls +(0.5,0) and +(-0.5,0).. (2,3) -- (3,3);
\end{scope}

\begin{scope}[xshift=14cm, yshift=-1.8cm,rotate=0,scale=1, style=dotted]
\draw[color=red] (0,4) -- (1,4) ..controls +(0.5,0) and +(-0.5,0).. (2,3) -- (3,3) -- (6,3);
\end{scope}

\begin{scope}[xshift=4cm, yshift=-1.8cm,rotate=0,scale=1, style=dotted]
\draw [color=red](0,3) -- (1,3);
\node (ca) at (1,3) { };
\node (cc) at (1.5,3.5) { };
\node (cb) at (2,4) { };
\draw[color=red] (ca.center) .. controls (ca.2 north east) and (cc.2 south west) .. (cc);
\draw[color=red] (cc) .. controls (cc.2 north east) and (cb.2 south west) .. (cb.center);
\draw[color=red] (2,4) -- (3,4);
\end{scope}

\begin{scope}[xshift=11cm, yshift=-0.8cm,rotate=0,scale=1, style=dotted]
\draw [color=red](0,3) -- (1,3);
\node (ca) at (1,3) { };
\node (cc) at (1.5,3.5) { };
\node (cb) at (2,4) { };
\draw[color=red] (ca.center) .. controls (ca.2 north east) and (cc.2 south west) .. (cc);
\draw[color=red] (cc) .. controls (cc.2 north east) and (cb.2 south west) .. (cb.center);
\draw[color=red] (2,4) -- (3,4);
\end{scope}

\begin{scope}[xshift=7cm, yshift=-3.2cm,rotate=0,scale=1, style=dotted]
\draw [color=red](0,3) -- (1,3);
\node (ca) at (1,3) { };
\node (cc) at (1.5,3.5) { };
\node (cb) at (2,4) { };
\draw[color=red] (ca.center) .. controls (ca.2 north east) and (cc.2 south west) .. (cc);
\draw[color=red] (cc) .. controls (cc.2 north east) and (cb.2 south west) .. (cb.center);
\draw[color=red] (2,4) -- (3,4);
\end{scope}

\begin{scope}[xshift=14cm, yshift=-2.2cm,rotate=0,scale=1, style=dotted]
\draw[color=red] (0,3) -- (1,3);
\node (ca) at (1,3) { };
\node (cc) at (1.5,3.5) { };
\node (cb) at (2,4) { };
\draw[color=red] (ca.center) .. controls (ca.2 north east) and (cc.2 south west) .. (cc);
\draw[color=red] (cc) .. controls (cc.2 north east) and (cb.2 south west) .. (cb.center);
\draw[color=red] (2,4) -- (3,4);
\end{scope}

\begin{scope}[xshift=1cm, yshift=-1cm,rotate=0,scale=1]
\draw [color=blue](0,4) -- (1,4) ..controls +(0.5,0) and +(-0.5,0).. (2,3) -- (3,3);
\draw (0,3) -- (1,3);
\node (ca) at (1,3) { };
\node (cc) at (1.5,3.5) { };
\node (cb) at (2,4) { };
\draw (ca.center) .. controls (ca.2 north east) and (cc.2 south west) .. (cc);
\draw (cc) .. controls (cc.2 north east) and (cb.2 south west) .. (cb.center);
\draw (2,4) -- (3,4);
\end{scope}

\begin{scope}[xshift=4cm, yshift=-2cm,rotate=0,scale=1]
\draw [color=blue](0,4) -- (1,4) ..controls +(0.5,0) and +(-0.5,0).. (2,3) -- (3,3);
\draw [color=blue](0,3) -- (1,3);
\node (ca) at (1,3) { };
\node (cc) at (1.5,3.5) { };
\node (cb) at (2,4) { };
\draw [color=blue](ca.center) .. controls (ca.2 north east) and (cc.2 south west) .. (cc);
\draw [color=blue](cc) .. controls (cc.2 north east) and (cb.2 south west) .. (cb.center);
\draw [color=blue](2,4) -- (3,4);
\end{scope}

\begin{scope}[xshift=7cm, yshift=-3cm,rotate=0,scale=1]
\draw [color=blue](0,4) -- (1,4) ..controls +(0.5,0) and +(-0.5,0).. (2,3) -- (3,3);
\draw (0,3) -- (1,3);
\node (ca) at (1,3) { };
\node (cc) at (1.5,3.5) { };
\node (cb) at (2,4) { };
\draw (ca.center) .. controls (ca.2 north east) and (cc.2 south west) .. (cc);
\draw (cc) .. controls (cc.2 north east) and (cb.2 south west) .. (cb.center);
\draw (2,4) -- (3,4);
\end{scope}

\begin{scope}[xshift=11cm, yshift=-1cm,rotate=0,scale=1]
\draw (0,4) -- (1,4) ..controls +(0.5,0) and +(-0.5,0).. (2,3) -- (3,3);
\draw [color=blue](0,3) -- (1,3);
\node (ca) at (1,3) { };
\node (cc) at (1.5,3.5) { };
\node (cb) at (2,4) { };
\draw [color=blue](ca.center) .. controls (ca.2 north east) and (cc.2 south west) .. (cc);
\draw [color=blue](cc) .. controls (cc.2 north east) and (cb.2 south west) .. (cb.center);
\draw [color=blue](2,4) -- (3,4);
\end{scope}

\begin{scope}[xshift=14cm, yshift=-2cm,rotate=0,scale=1]
\draw (0,4) -- (1,4) ..controls +(0.5,0) and +(-0.5,0).. (2,3) -- (3,3);
\draw (0,3) -- (1,3);
\node (ca) at (1,3) { };
\node (cc) at (1.5,3.5) { };
\node (cb) at (2,4) { };
\draw (ca.center) .. controls (ca.2 north east) and (cc.2 south west) .. (cc);
\draw (cc) .. controls (cc.2 north east) and (cb.2 south west) .. (cb.center);
\draw (2,4) -- (3,4);
\end{scope}

\begin{scope}[xshift=17cm, yshift=-3cm,rotate=0,scale=1]
\draw (0,4) -- (1,4) ..controls +(0.5,0) and +(-0.5,0).. (2,3) -- (3,3);
\draw [color=blue](0,3) -- (1,3);
\node (ca) at (1,3) { };
\node (cc) at (1.5,3.5) { };
\node (cb) at (2,4) { };
\draw [color=blue](ca.center) .. controls (ca.2 north east) and (cc.2 south west) .. (cc);
\draw [color=blue](cc) .. controls (cc.2 north east) and (cb.2 south west) .. (cb.center);
\draw [color=blue](2,4) -- (3,4);
\end{scope}
\end{scope}
\end{scope}
\end{scope}
\end{scope}
\end{tikzpicture}
\caption{The torus link $T(4,2)$ as a $(2,4)$ cable on $T(2,1)$} \label{fig t(4,2)}
\end{figure}
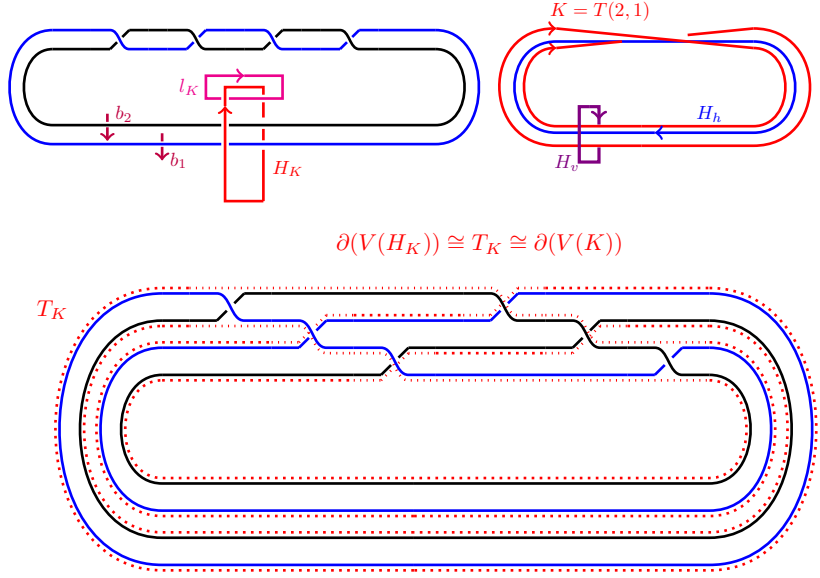

In Figure \ref{fig t(4,2)}, we draw a torus link $T(e,epq)$ inside a solid torus $S^3 \setminus V(H_K)$, the link $T(p,q) \cup H_v \cup H_h$, and the torus link $T(ep,eq)$ which is a $(e,epq)$-cable on $T(p,q)$ (we did not draw $H_v$ and $H_h$ in the third part in order to make the figure easier to read). Here $p =2, q=1, e=2$.
 One can see the  torus $T_K$ (drawn with red dotted lines) that separates $M_{T(4,2)}$ in the disjoint union of the exterior of the torus knot $T(2,1)$ in $S^3$ and the exterior of the torus link $T(2,4) \cup H_K$. This torus $T_K$ is the boundary of a tubular neighbourhood of $K=T(p,q)$. A preferred longitude $l_K$ of $K$ is drawn on the figure. We identify $S^3 \setminus V(H_K)$ to the solid torus $V(K)$; the component $H_K$ looks like a preferred meridian of $K$.
Let $M = M_L$ denote the exterior of  $L = T(ep,eq) \cup H_v \cup H_h$, $A = S^3 \setminus V(K \cup H_v \cup H_h)$ and $B = S^3 \setminus V(T(e,epq) \cup H_K)$ (in Figure \ref{fig t(4,2)}, $A$  is the exterior of the drawing up right and $B$ of the one up left). We see that $M$ is the toroidal gluing of $A$ and $B$ along their intersection $T_K = A \cap B$. The following diagrams are commutative:

\begin{tikzpicture}
[description/.style={fill=white,inner sep=2pt}] 
\matrix(a)[matrix of math nodes, row sep=2em, column sep=1.5em, text height=1.5ex, text depth=0.25ex] 
{ & A\\ T_K & & M\\ & B\\}; 
\path[->](a-2-1) edge node[below]{$I_B$} (a-3-2); 
\path[->](a-2-1) edge node[above]{$I_A$} (a-1-2);  
\path[->](a-1-2) edge node[above]{$J_A$} (a-2-3); 
\path[->](a-3-2) edge node[below]{$J_B$} (a-2-3); 
\path[->](a-2-1) edge node[above]{$I$} (a-2-3);  

\begin{scope}[xshift=6cm,rotate=0,scale=1]
[description/.style={fill=white,inner sep=2pt}] 
\matrix(a)[matrix of math nodes, row sep=2em, column sep=2em, text height=1.5ex, text depth=0.25ex] 
{ & \pi_1(A)\\ \pi_1(T_K) & & \pi_1(M) = G_{L} & G\\ & \pi_1(B) &  \Z^{e+2} & \Z\\}; 
\path[->](a-2-1) edge node[below]{$i_B$} (a-3-2); 
\path[->](a-2-1) edge node[above]{$i_A$} (a-1-2);  
\path[->](a-1-2) edge node[above]{$j_A$} (a-2-3); 
\path[->](a-3-2) edge node[below]{$j_B$} (a-2-3); 
\path[->](a-2-1) edge node[above]{$i$} (a-2-3);
\path[->](a-2-3) edge node[above]{$\gamma$} (a-2-4);
\path[->](a-2-3) edge node[right]{$\alpha_{L}$} (a-3-3);
\path[->](a-3-3) edge node[below]{$(n_1,\ldots,n_{e+2})$} (a-3-4);
\path[->, dotted](a-2-4) edge  (a-3-4);
\end{scope}
\end{tikzpicture}

As in the previous section, let $T$ be the torus on which $K$ is drawn, and $\delta$ a simple closed curve that generates the fundamental group of $T \setminus V(K)$. The curve $\delta$ is once again locally parallel to the components of $T(ep,eq)$.

\begin{proposition} \label{prop_ep,eq_inside_thick}
Let $e \geqslant 2$. The $L^2$-Alexander torsion for the exterior of the  link 
 $L = T(ep,eq) \cup H_v \cup H_h$  is non-zero for all $t>0$ and for all admissible triples 
 $(G_L, (n_1,\ldots,n_{e},n_{e+1},n_{e+2}) \circ \alpha_L, \gamma)$ such that the homotopy class of the curve $\delta$ is sent by $\gamma$ to an element of infinite order. One has:
$$
 T^{(2)}_{L, (n_1,\ldots,n_{e},n_{e+1},n_{e+2})}(\gamma)(t)
 \ \dot{=} \  \max(1,t)^{e|pq(n_1+ \ldots +n_e) + p n_{e+1} + q n_{e+2}|}.
$$
\end{proposition}

\begin{proof}
Let $t>0$. Let $\phi$ denote the homomorphism $(n_1, \ldots, n_{e+2}) \circ \alpha_L$. We assume that the homotopy class of $\delta$ is sent by $\gamma$ to an element of infinite order.
First, as the cabling torus $T_K$ is the boundary of a tubular neighbourhood $V(K)$ of $K = T(p,q)$ and contains such a curve $\delta$, the torus $T_K$ has thus infinite image under $\gamma$, therefore 
$C^{(2)}_*(T_K,\phi \circ i, \gamma \circ i,t)$ is weakly acyclic and of determinant class and its $L^2$-torsion is $1$, by Theorem \ref{torsion torus}.

Secondly, the curve $\lambda \mu^k = \lambda \mu^{pq}$ of Proposition \ref{prop_e,ek_inside_torus} is ambient isotopic to $\delta$, thus it is sent by $\gamma$ to an element of infinite order (in Figure \ref{fig t(4,2)}, $\lambda$ is written $l_K$, and $\mu$ is parallel to $H_K$), therefore
$C^{(2)}_*(B,\phi \circ j_B, \gamma \circ j_B,t)$ is weakly acyclic and of determinant class, and
\begin{align*}
 T^{(2)}(B,\phi \circ j_B, \gamma \circ j_B)(t) 
& =  
 T^{(2)}\left (M_{L'},(n_1, \ldots, n_{e+2}) \circ \alpha_L \circ j_B, \gamma \circ j_B\right )(t)  \\
 & =
 T^{(2)}\left (M_{L'},(n_1, \ldots, n_e, p n_{e+1} + q n_{e+2}) \circ \alpha_{L'}, \gamma \circ j_B\right )(t)  \\
& \ \dot{=} \  \max(1,t)^{(e-1)|p n_{e+1} + q n_{e+2} + pq(n_1 + \ldots + n_e)|},
\end{align*}
where $L' = T(e,epq) \cup H_K$.

Finally, the last piece of the toroidal gluing is $A =M_{H} \setminus V(K)$, which corresponds to the case $e=1$ of the previous Section \ref{sec:tpqhh}; from the assumption on $\delta$, it follows from Proposition \ref{prop_p,q_inside_thick} that
$C^{(2)}_*(A,\phi \circ j_A, \gamma \circ j_A,t)$ is weakly acyclic and of determinant class, and
\begin{align*}
 T^{(2)}(A,\phi \circ j_A, \gamma \circ j_A)(t) 
& =  
 T^{(2)}\left (M_{L''},(n_1, \ldots, n_{e+2}) \circ \alpha_L \circ j_A, \gamma \circ j_A\right )(t)  \\
 & =
 T^{(2)}\left (M_{L''},(n_1+ \ldots+ n_e, n_{e+1}, n_{e+2}) \circ \alpha_{L''}, \gamma \circ j_A\right )(t)  \\
& \ \dot{=} \  \max(1,t)^{|p n_{e+1} + q n_{e+2} + pq(n_1 + \ldots + n_e)|},
\end{align*}
where $L''=T(p,q) \cup H_v \cup H_h$.
It follows then from Proposition \ref{prop t2 JSJ} that
$C^{(2)}_*(M_L,\phi, \gamma,t)$ is weakly acyclic and of determinant class, and
\begin{align*}
&T^{(2)}_{L, (n_1,\ldots,n_{e},n_{e+1},n_{e+2})}(\gamma)(t) \\
&= T^{(2)}(A,\phi \circ j_A, \gamma \circ j_A)(t) \cdot T^{(2)}(B,\phi \circ j_B, \gamma \circ j_B)(t) \\
& \ \dot{=}  \max(1,t)^{(e-1)|p n_{e+1} + q n_{e+2} + pq(n_1 + \ldots + n_e)|} \cdot \max(1,t)^{|p n_{e+1} + q n_{e+2} + pq(n_1 + \ldots + n_e)|} \\
& = \max(1,t)^{e |p n_{e+1} + q n_{e+2} + pq(n_1 + \ldots + n_e)|}.
\end{align*}
\end{proof}

\subsection{The link $T(ep,eq) \cup H_v$}
\label{sec:tepeqh}

The link $L = L_1 \cup \ldots \cup L_e \cup H_v =  T(ep,eq) \cup H_v$ is obtained from $L' = T(ep,eq) \cup H_v \cup H_h$ by deleting the component $H_h$, therefore $M_L$ is obtained from $M_{L'}$ by a $\infty$-surgery on the boundary component of $H_h$. This helps us compute the $L^2$-Alexander torsions of $L$. Let $\lambda_h$ be the homotopy class of $H_h$ in $M_L$ and $\delta$ the simple closed curve locally parallel to the strands of $T(ep,eq)$ as in the previous section. The epimorphism $Q: G_{L'} \to G_{L}$ corresponds to the trivialization of the curve $\lambda_h$.

\begin{proposition} \label{prop_ep,eq_inside}
For the exterior of the  link
 $L = T(ep,eq) \cup H_v$,  The $L^2$-Alexander torsion is non-zero for all admissible triples 
 $(G_L, (n_1,\ldots,n_{e+1}) \circ \alpha_L, \gamma)$ such that
$\gamma(\delta)$ and $\gamma(\lambda_h)$ have infinite order in $G$
 and for all $t>0$. One has:
$$ T^{(2)}_{L, (n_1,\ldots,n_{e+1})}(\gamma)(t)
= \max(1,t)^{(e|p|-1)|n_{e+1}+q(n_1+ \ldots +n_e)|}.$$
\end{proposition}

\begin{proof}
We will use Theorem \ref{thm surgery forget}. Here $\lambda_h$ corresponds to the curve $\lambda$ in the assumptions of Theorem \ref{thm surgery forget}. Since
$\gamma(\delta)$ has infinite order in $G$, it follows from Proposition \ref{prop_ep,eq_inside_thick} that
$C^{(2)}_*(M_{L'},(n_1,\ldots,n_{e+1},0) \circ \alpha_{L'}, \gamma \circ Q,t)$ is weakly acyclic and of determinant class, and 
$$ T^{(2)}_{L',(n_1,\ldots,n_{e+1},0)}(\gamma \circ Q)(t)
\ \dot{=} \  \max(1,t)^{e|pq(n_1+ \ldots +n_e) + p n_{e+1}|}.$$
Since $\gamma(\lambda_h)$ has infinite order in $G$, then
$C^{(2)}_*(M_{L},(n_1,\ldots,n_{e+1}) \circ \alpha_{L}, \gamma,t)$ is weakly acyclic and of determinant class by  Theorem \ref{thm surgery forget}, and 
\begin{align*}
T^{(2)}_{L, (n_1,\ldots,n_{e+1})}(\gamma)(t) 
&\ \dot{=} \  \dfrac{T^{(2)}_{L',(n_1,\ldots,n_{e+1},0)}(\gamma \circ Q)(t)}{\max(1,t)^{|\mathrm{lk}(L_1,H_h) n_1 + \ldots + \mathrm{lk}(L_e,H_h) n_e + \mathrm{lk}(H_v,H_h) n_{e+1}|}} \\
&\ \dot{=} \  \dfrac{\max(1,t)^{e|pq(n_1+ \ldots +n_e) + p n_{e+1}|}}
{\max(1,t)^{|q n_1 + \ldots + q n_e +  n_{e+1}|}} \\
& = \max(1,t)^{(e|p|-1)|n_{e+1}+q(n_1+ \ldots +n_e)|}.
\end{align*} 
 \end{proof}

\subsection{The torus link $T(ep,eq)$}
\label{sec:tepeq}

Now we can compute $L^2$-Alexander torsions for general torus links of the form $L = T(ep,eq)$, where $e \geqslant 2$ is an integer and $p,q$ are relatively prime integers.
The link $T(ep,eq)$ is obtained by $\infty$-surgery from $T(ep,eq) \cup H_v$ on the component $H_v$. 
The epimorphism $Q: G_{T(ep,eq) \cup H_v} \to G_{T(ep,eq)}$ corresponds to the trivialization of the curve $\lambda_v$.
Let $\delta$ and $\lambda_h$ be as in the previous sections, and let $\lambda_v$ denote the homotopy class of $H_v$ in $G_{T(ep,eq)}$. Note that the fundamental group of the torus $T$ (on which $T(ep,eq)$ is drawn) is generated by classes of curves homotopic to $\lambda_h$ and $\lambda_v$. Thus the equality
$ \delta = \lambda_h^p \lambda_v^q$
stands in $G_{T(ep,eq)}$.
This equality and the fact  that $\lambda_h \lambda_v = \lambda_v \lambda_h$ imply that, for any homomorphism $\gamma: G_{T(ep,eq)} \to G$, if two elements of $\{ \gamma(\delta), \gamma(\lambda_h), \gamma(\lambda_v) \}$ are of infinite order, then the third is of infinite order as well.

\begin{proposition} \label{prop_ep,eq}
 The $L^2$-Alexander torsion for the exterior of the torus link 
 $L = T(ep,eq)$  is non-zero for all admissible triples 
 $(G_L, (n_1,\ldots,n_e) \circ \alpha_L, \gamma)$ such that two of the three elements
$\gamma(\delta), \gamma(\lambda_h), \gamma(\lambda_v)$
 have infinite order in $G$,
 and for all $t>0$. One has:
$$ T^{(2)}_{L, (n_1,\ldots,n_e)}(\gamma)(t)
\ \dot{=} \ \max(1,t)^{|n_1+ \ldots +n_e|\left (e|p||q| -|p| -|q|\right )}.$$
\end{proposition}

This theorem generalises the computation of the $L^2$-Alexander invariants of torus knots done in \cite[Proposition 6.2]{DW}.

\begin{proof}
Here $\lambda_v$ corresponds to the curve $\lambda$ in the assumptions of Theorem \ref{thm surgery forget}. Since
$\gamma(\delta)$, $\gamma(\lambda_h)$ and $\gamma(\lambda_v)$ have infinite order in $G$, the result follows from Proposition \ref{prop_ep,eq_inside} and Theorem \ref{thm surgery forget} exactly like in the proof of Proposition \ref{prop_ep,eq_inside}.
\end{proof}

\subsection{General cabling formulas}
\label{sec:cabling}

We can now prove a general cabling formula for $L^2$-Alexander torsions as a consequence of Proposition \ref{prop_ep,eq_inside}.
Let $L = L_1 \cup \ldots \cup L_{c+1}$ a link in $S^3$, and $L' = L_1 \cup \ldots \cup L_{c} \cup L'_{c+1} \cup \ldots \cup L'_{c+e}$ the link obtained by cabling the component $L_{c+1}$ by the torus link $T(ep,eq)$ with $p,q$ two relatively prime integers.

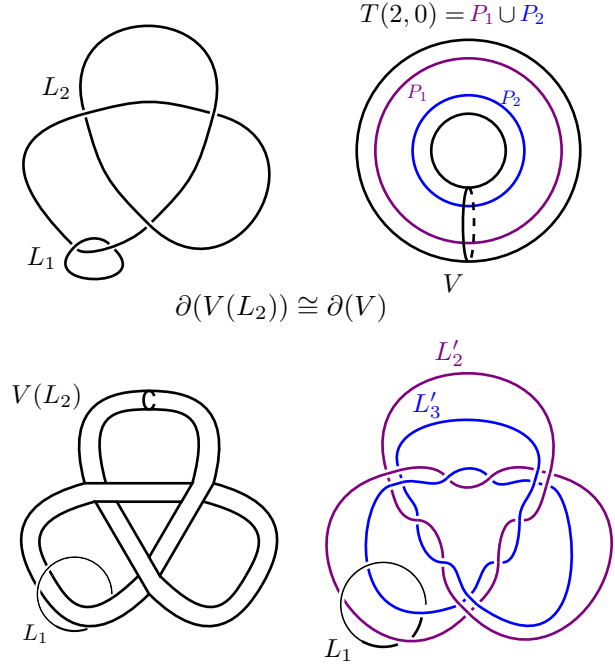
\begin{figure}[!h]
\centering
\begin{tikzpicture} [every path/.style={string } , every node/.style={transform shape , knot crossing , inner sep=1.7 pt } ] 
\begin{scope}[scale=0.7]
\begin{scope}[xshift=-1cm,scale=0.7]
	\node[rotate=150] (tl) at (-1.73,1) {};
	\node[rotate=30] (tr) at (1.73, 1 ) { } ;
	\node[rotate=0] (b) at (0 ,-2) { } ;
	\coordinate (n) at (-2,1.2) ;
	\draw[scale=2] (n) node[above left]{$L_2$} ;
	\draw[scale=2] (-1.2,-1.7) node[above left]{$L_1$} ;	
	
	\node[rotate=0] (a) at (-1 ,-2.6) { } ;
	\node[rotate=0] (c) at (-2 ,-2.6) { } ;
\draw (tl.center) .. controls (tl.16 south west) and (tr.16 north west) .. (tr) ;
\draw (tl) .. controls (tl.32 south east) and (tr.32 north east) .. (tr.center) ;
\draw (tl) .. controls (tl.16 north west) and (b.16 north west) .. (b.center) ;
\draw (tl.center) .. controls (tl.16 north east) and (c.32 north west) .. (c) ;
\draw (b) .. controls (b.8 south west) and (c.4 south east) .. (c) ;
\draw (b) .. controls (b.16 north east) and (tr.16 south west) .. (tr.center) ;
\draw (b.center) .. controls (b.32 south east) and (tr.32 south east) .. (tr) ;
\draw (a) .. controls (a.16 south east) and (c.16 south west) .. (c.center) ;
\draw (c.center) .. controls (c.4 north east) and (a.4 north west) .. (a) ;
\end{scope}
\begin{scope}[xshift=5cm,scale=0.7]
	\coordinate (n) at (-1,1.2) ;
	\coordinate (m) at (-2.6,1.5) ;
	\draw[scale=2] (0,-2) node[above left]{$V$} ;
	\coordinate (O) at (0,0) ;
	\coordinate (A) at (0.75,1.3) ;
	\coordinate (B) at (1.25,2.17) ;
	\coordinate (C) at (-0.75,1.3) ;
	\coordinate (D) at (-1.25,2.17) ;	
	\coordinate (G) at (0,-1) ;
	\coordinate (H) at (0,-3) ;
	\coordinate (I) at (-0.15,2.1) ;
	\draw (O) circle (1) ;
	\draw (O) circle (3) ;	
	\draw[color=blue] (O) circle (1.5) ;
	\draw [color=violet](O) circle (2.5) ;	
	\draw (G) ..controls +(-0.2,0) and +(-0.2,0).. (H) ;
	\draw [dashed] (G) ..controls +(0.2,0) and +(0.2,0).. (H) ;
		\draw[scale=1.5] (n) node[violet, above left]{$P_1$} ;
				\draw[scale=1.5] (1,0.7) node[blue, above left]{$P_2$} ;
		\draw[scale=2] (0.7,0.9) node[above left]{$T(2,0) = \ \ \ \cup \ \ \ $} ;
		\draw[scale=1.8] (0.45,1.8) node[violet, above left]{$P_1$} ;
				\draw[scale=1.8] (1.2,1.8) node[blue, above left]{$P_2$} ;
\end{scope}
\draw[scale=1.5] (2,-3) node[above left]{$ \partial (V(L_2)) \cong \partial (V) $} ;
\begin{scope}[xshift=-1cm,yshift=-7cm,scale=0.7]
	\coordinate (n) at (-1.9,2.5) ;
	\draw[scale=2] (n) node[above left]{$ V(L_2) $} ;
	\coordinate (O) at (0,0) ;
	\coordinate (A) at (-0.87,0.5) ;
	\coordinate (B) at (-1.73,1) ;
	\coordinate (C) at (-1.16,1) ;
	\coordinate (D) at (0.87,0.5) ;
	\coordinate (E) at (1.73,1) ;
	\coordinate (F) at (1.16,1) ;
	\coordinate (G) at (0,3) ;
	\coordinate (H) at (0,3.5) ;
	\draw (A) -- (D) -- (F) -- (B) ;
	\draw (F) ..controls +(0.5*1,0.5*3.73) and +(1,0).. (G) ;
	\draw (C) ..controls +(-0.5*1,0.5*3.73) and +(-1,0).. (G) ;
	\draw (E) ..controls +(0.6*1,0.6*3.73) and +(1,0).. (H) ;
	\draw (B) ..controls +(-0.6*1,0.6*3.73) and +(-1,0).. (H) ;
	\draw (G) ..controls +(-0.2,0) and +(-0.2,0).. (H) ;
	\draw [dashed] (G) ..controls +(0.2,0) and +(0.2,0).. (H) ;
\begin{scope}[rotate=120]
	\coordinate (O) at (0,0) ;
	\coordinate (A) at (-0.87,0.5) ;
	\coordinate (B) at (-1.73,1) ;
	\coordinate (C) at (-1.16,1) ;
	\coordinate (D) at (0.87,0.5) ;
	\coordinate (E) at (1.73,1) ;
	\coordinate (F) at (1.16,1) ;	
	\coordinate (G) at (0,3) ;
	\coordinate (H) at (0,3.5) ;
	\draw (A) -- (D) -- (F) -- (B) ;
	\draw (F) ..controls +(0.5*1,0.5*3.73) and +(1,0).. (G) ;
	\draw (C) ..controls +(-0.5*1,0.5*3.73) and +(-1,0).. (G) ;
	\draw (E) ..controls +(0.6*1,0.6*3.73) and +(1,0).. (H) ;
	\draw (B) ..controls +(-0.6*1,0.6*3.73) and +(-1,0).. (H) ;
\end{scope}
\begin{scope}[rotate=240]
	\coordinate (O) at (0,0) ;
	\coordinate (A) at (-0.87,0.5) ;
	\coordinate (B) at (-1.73,1) ;
	\coordinate (C) at (-1.16,1) ;
	\coordinate (D) at (0.87,0.5) ;
	\coordinate (E) at (1.73,1) ;
	\coordinate (F) at (1.16,1) ;	
	\coordinate (G) at (0,3) ;
	\coordinate (H) at (0,3.5) ;
	\draw (A) -- (D) -- (F) -- (B) ;
	\draw (F) ..controls +(0.5*1,0.5*3.73) and +(1,0).. (G) ;
	\draw (C) ..controls +(-0.5*1,0.5*3.73) and +(-1,0).. (G) ;
	\draw (E) ..controls +(0.6*1,0.6*3.73) and +(1,0).. (H) ;
	\draw (B) ..controls +(-0.6*1,0.6*3.73) and +(-1,0).. (H) ;
\end{scope}
\draw[color=black,  thick] (-1,-2) arc (0:290:1);
\draw[color=white,  thick] (-0.95,-2) arc (0:240:1.05);
\draw[color=white,  thick] (-1.05,-2) arc (0:240:0.95);
	\draw[scale=1.75] (-1.5,-2) node[black, above left]{$L_1$} ;
\end{scope}
\begin{scope}[xshift=5cm,yshift=-7cm,scale=0.8,color=blue]
	\coordinate (n) at (-1,2) ;
	\node[rotate=150] (g1) at (-1.63,1.18) {};
	\node[rotate=150] (g2) at (-1.91,1.1) {};
	\node[rotate=150] (g3) at (-1.83,0.82) {};
	\node[rotate=150] (g4) at (-1.55,0.9) {};
	\node[rotate=30] (d1) at (1.63, 1.18 ) { } ;
	\node[rotate=30] (d2) at (1.55, 0.9 ) { } ;	
	\node[rotate=30] (d3) at (1.83, 0.82 ) { } ;
	\node[rotate=30] (d4) at (1.91, 1.1 ) { } ;
	\node[rotate=0] (b1) at (0 ,-1.8) { } ;
	\node[rotate=0] (b2) at (-0.2 ,-2) { } ;	
	\node[rotate=0] (b3) at (0 ,-2.2) { } ;
	\node[rotate=0] (b4) at (0.2 ,-2) { } ;
	\node[rotate=0] (h) at (0 ,2.5) { } ;	
\draw (g1) .. controls (g1.16 south east) and (d1.16 north east) .. (d1.center) ;
\draw [color=violet](g2) .. controls (g2.32 south east) and (d4.32 north east) .. (d4.center) ;
\draw [color=violet](g2.center) .. controls (g2.32 north east) and (b3.32 south west) .. (b3) ;
\draw (g3.center) .. controls (g3.8 north east) and (-3,-3) .. (b2) ;
\draw [color=violet](g2.center) -- (g1.center) ;
\draw (g3.center) -- (g4.center) ;
\draw (g1)-- (g4);
\draw [color=violet](g2) -- (g3) ;
\draw [color=violet](b3.center) .. controls (b3.32 south east) and (d4.32 south east) .. (d4) ;
\draw (b4.center) .. controls (3,-4) and (d3.8 south east) .. (d3) ;
\draw (d2.center) -- (d1.center) ;
\draw [color=violet](d3.center) -- (d4.center) ;
\draw [color=violet](d1)-- (d4);
\draw (d2) -- (d3) ;
\draw (b1.center) -- (b4.center) ;
\draw [color=violet](b2.center) -- (b3.center) ;
\draw (b1)-- (b2);
\draw [color=violet](b3) -- (b4) ;
	\node (tm1) at (-0.5,1) {};
	\node (tm2) at (0.5,1) {};
\draw [color=violet](g1.center) .. controls (g1.4 south west) and (tm1.4 north west) .. (tm1) ;
\draw [color=violet](tm1) .. controls (tm1.4 south east) and (tm2.4 south west) .. (tm2.center) ;
\draw [color=violet](tm2.center) .. controls (tm2.4 north east) and (d1.4 north west) .. (d1) ;
\draw (g4.center) .. controls (g4.4 south west) and (tm1.4 south west) .. (tm1.center) ;
\draw (tm1.center) .. controls (tm1.4 north east) and (tm2.4 north west) .. (tm2) ;
\draw (tm2) .. controls (tm2.4 south east) and (d2.4 north west) .. (d2) ;
	\node (tg1) at (-1.2,0) {};
	\node (tg2) at (-0.6,-1) {};
\draw (g4) .. controls (g4.4 north west) and (tg1.4 north) .. (tg1) ;
\draw (tg1) .. controls (tg1.4 south) and (tg2.4 west) .. (tg2.center) ;
\draw (tg2.center) .. controls (tg2.4 east) and (b1.4 north west) .. (b1.center) ;
\draw [color=violet](g3) .. controls (g3.4 north west) and (tg1.4 west) .. (tg1.center) ;
\draw [color=violet](tg1.center) .. controls (tg1.4 east) and (tg2.4 north) .. (tg2) ;
\draw [color=violet](tg2) .. controls (tg2.4 south) and (b2.4 north west) .. (b2.center) ;
	\node (td1) at (1.2,0) {};
	\node (td2) at (0.7,-1) {};
\draw [color=violet](d3.center) .. controls (d3.4 south west) and (td1.4 east) .. (td1) ;
\draw [color=violet](td1) .. controls (td1.4 west) and (td2.4 north) .. (td2.center) ;
\draw [color=violet](td2.center) .. controls (td2.4 south) and (b4.4 north east) .. (b4) ;
\draw (d2.center) .. controls (d2.4 south west) and (td1.4 north) .. (td1.center) ;
\draw (td1.center) .. controls (td1.4 south) and (td2.4 east) .. (td2) ;
\draw (td2) .. controls (td2.4 west) and (b1.4 north east) .. (b1) ;
\draw[color=black] (-1,-2) arc (0:290:1);
\draw[color=white] (-0.95,-2) arc (0:240:1.05);
\draw[color=white] (-1.05,-2) arc (0:240:0.95);
\draw[color=black] (-1.1,-2.3) arc (-10:-40:1) ;
	\draw[scale=1.75] (-1.5,-2) node[black, above left]{$L_1$} ;
	\draw[scale=1.75] (0,2) node[violet, above left]{$L'_2$} ;
		\draw[scale=1.75] (-0.3,1.3) node[blue, above left]{$L'_3$} ;
\end{scope}
\end{scope}
\end{tikzpicture}
\caption{The $(2,0)$-cabling on the second component of $L = L_1 \cup L_2$} \label{fig cable 2,0}
\end{figure}

Then $M = M_{L'} = S^3 \setminus V(L')$ is obtained by a toroidal gluing of $ A = M_L = S^3 \setminus V(L)$ and 
$B =(S^1 \times D^2) \setminus V(T(ep,eq)) \cong S^3 \setminus V(T(ep,eq)\cup H_v)$
 between the components $L_{c+1}$ and $H_{v}$. 
 Let $n_1, \ldots, n_{c+e} \in \Z$. Let $\gamma: G_{L'} \to G$ be a group homomorphism such that $(G_{L'}, (n_1, \ldots, n_{c+e}) \circ \alpha_{L'}, \gamma)$ is an admissible triple. Let $t>0$.

\begin{tikzpicture}
[description/.style={fill=white,inner sep=2pt}] 
\matrix(a)[matrix of math nodes, row sep=2em, column sep=1.5em, text height=1.5ex, text depth=0.25ex] 
{ & A\\ T & & M\\ & B\\}; 
\path[->](a-2-1) edge node[below]{$I_B$} (a-3-2); 
\path[->](a-2-1) edge node[above]{$I_A$} (a-1-2);  
\path[->](a-1-2) edge node[above]{$J_A$} (a-2-3); 
\path[->](a-3-2) edge node[below]{$J_B$} (a-2-3); 
\path[->](a-2-1) edge node[above]{$I$} (a-2-3);  

\begin{scope}[xshift=6cm,rotate=0,scale=1]
[description/.style={fill=white,inner sep=2pt}] 
\matrix(a)[matrix of math nodes, row sep=2em, column sep=1.5em, text height=1.5ex, text depth=0.25ex] 
{ & \pi_1(A) = G_L\\ \pi_1(T) & & \pi_1(M) = G_{L'} & G\\ & \pi_1(B)  & \Z^{e+1} & \Z\\}; 
\path[->](a-2-1) edge node[below]{$i_B$} (a-3-2); 
\path[->](a-2-1) edge node[above]{$i_A$} (a-1-2);  
\path[->](a-1-2) edge node[above]{$j_A$} (a-2-3); 
\path[->](a-3-2) edge node[below]{$j_B$} (a-2-3); 
\path[->](a-2-1) edge node[above]{$i$} (a-2-3);
\path[->](a-2-3) edge node[above]{$\gamma$} (a-2-4);
\path[->](a-2-3) edge node[left]{$\alpha_{L'}$} (a-3-3);
\path[->](a-3-3) edge node[below]{$(n_1,\ldots,n_{c+e})$} (a-3-4);
\path[->, dotted](a-2-4) edge  (a-3-4);
\end{scope}
\end{tikzpicture}

Let 
$N = n_{c+1} + \ldots + n_{c+e}$
and
$\ell = \sum_{i=1}^c \mathrm{lk}(L_i,L_{c+1}) n_i.$
To clarify the notations, let us look a the example in Figure \ref{fig cable 2,0}.
The link $L$ has two components ($c=1$), $L_1$ which is unknotted and $L_2$ which is a trefoil, with linking number $\mathrm{lk}(L_1,L_2)=1$. We do a $(2,0)$-cabling on $L_2$ (thus $e=2, p=1, q=0$), and the resulting link $L'$ has $3$ components. We glue the tori $\partial( V(L_2))$ and $\partial (V)$ such that a meridian of $L_2$ is identified with $m_V$ the meridian of $V$ that circles both components of $T(2,0)$.
Here $N= n_2+n_3$ and $\ell= n_1$.

\begin{theorem} \label{thm_torsion_cabling}
Assume that 
\begin{itemize}
\item $C^{(2)}_*(M_L, (n_1, \ldots,n_c, p N) \circ \alpha_L, \gamma \circ j_A,t)$ is weakly acyclic and of determinant class,
\item $C^{(2)}_*(M_{T(ep,eq)\cup H_v},(n_{c+1}, \ldots, n_{c+e},\ell) \circ \alpha_{T(ep,eq)\cup H_v}, \gamma \circ j_B,t)$ is weakly acyclic and of determinant class,
\item $T = \partial(V) \cong \partial(V(L_{c+1}))$ has infinite image under $\gamma$,
\end{itemize}
then $C^{(2)}_*(M_{L'}, (n_1, \ldots, n_{c+e}) \circ \alpha_{L'}, \gamma,t)$ is weakly acyclic and of determinant class and
$$
T^{(2)}_{L', (n_1, \ldots, n_{c+e})}(\gamma)(t) 
\ \dot{=} \ 
T^{(2)}_{L, (n_1, \ldots,n_c, p N)}(\gamma \circ j_A)(t) \cdot
\max(1,t)^{(e|p|-1)|\ell+q N|}. 
$$
\end{theorem}

This theorem generalises the cabling formula for the $L^2$-Alexander invariant of knots proven in \cite[Theorem 4.3]{BA13}.

\begin{proof}
First let us prove that
$$ (n_1, \ldots, n_{c+e}) \circ \alpha_{L'} \circ j_A
= (n_1, \ldots,n_c, p N) \circ \alpha_L
$$
and that $$ (n_1, \ldots, n_{c+e}) \circ \alpha_{L'} \circ j_B
= (n_{c+1}, \ldots, n_{c+e},\ell) \circ \alpha_{T(ep,eq)\cup H_v}.
$$
The group $G_L = \pi_1(A)$ is generated by $m_1, \ldots m_{c+1}$, which are preferred meridians of $L_1, \ldots L_{c+1}$ in $M_L$. One has
$$ ((n_1, \ldots, n_{c+e}) \circ \alpha_{L'} \circ j_A)(m_i)
= ((n_1, \ldots,n_c, p N) \circ \alpha_L)(m_i)
$$ for $i=1, \ldots, c$ since
$L_1, \ldots, L_c$ are the $c$ first components of $L'$. The identity remains true for $i=c+1$, since 
$j_A(m_{c+1})$ circles the $e$ components $L'_{c+1}, \ldots, L'_{c+e}$ for a total of $p$ times and is unlinked with $L_1, \ldots, L_c$.
The group $G_B = \pi_1(B) = \pi_1(S^1 \times D^2 \setminus V(T(ep,eq)))$ is generated by $b_1, \ldots, b_e$ (which are preferred meridians of the components of $T(ep,eq)$) and $\lambda$ a longitude of the solid torus $S^1 \times D^2 \cong \partial V(L_{c+1})$. Note that $j_B(\lambda) = j_A(l_{c+1})$ in $M_{L'}$ where $l_{c+1}$ is a preferred longitude of 
$L_{c+1}$ in $M_L$. The identity
$$ (n_1, \ldots, n_{c+e}) \circ \alpha_{L'} \circ j_B
= (n_{c+1}, \ldots, n_{c+e},\ell) \circ \alpha_{T(ep,eq)\cup H_v}
$$ stands true when evaluated on each of the generators $b_i$, for $i=1 \ldots,e$ (both terms of the equality are immediately equal to $n_{c+i}$), and for $\lambda$ the second term is equal to $\ell$, and the first term is equal to
\begin{align*}
((n_1, \ldots, n_{c+e}) \circ \alpha_{L'})( j_B(\lambda))
&=((n_1, \ldots, n_{c+e}) \circ \alpha_{L'})( j_A(l_{c+1})) \\
&= ((n_1, \ldots,n_c, p N) \circ \alpha_L)(l_{c+1}) \\
&= n_1 \mathrm{lk}(L_1,L_{c+1})  + \ldots + n_c \mathrm{lk}(L_c,L_{c+1}) + 0
= \ell.
\end{align*}
We have established that the three different coefficents $\phi$ of the statement of the result were indeed compatible. Now, 
since the cabling torus $T = \partial(V(L_{c+1}))$ has infinite image under $\gamma$, the result follows from Proposition \ref{prop t2 JSJ} and Proposition \ref{prop_ep,eq_inside}.
 \end{proof}

\bibliographystyle{plain}

\bibliography{bibliothese}

\end{document}